\documentclass[12pt,leqno]{amsart}
\usepackage[latin1]{inputenc}
 \usepackage{amssymb,amsmath,amsthm,amscd,amsfonts}
\usepackage{color}
\usepackage{amscd,amssymb,verbatim,txfonts}
\usepackage{lmodern}
\usepackage{graphicx}
\usepackage{tikz-cd}
\usepackage{hyperref,cleveref,graphics,mathrsfs}

\newtheorem{thm}{Theorem}[section]
\newtheorem{cor}[thm]{Corollary}
\newtheorem{lem}[thm]{Lemma}
\newtheorem{prop}[thm]{Proposition}

\newtheorem{ex}[thm]{Example}
\newtheorem{prob}[thm]{Problem}

\theoremstyle{definition}
\newtheorem{remark}[thm]{Remark}
\newtheorem{question}[thm]{Question}

\numberwithin{equation}{section}

\newcommand{\ep}{\varepsilon}
\newcommand{\al}{\alpha}

\newcommand{\Om}{\Omega}

\newcommand{\vp}{\phi}

\newcommand{\la}{\langle}
\newcommand{\ra}{\rangle}

\newcommand{\ol}{\overline}
\newcommand{\ti}{\widetilde}
\newcommand{\wh}{\widehat}

\newcommand{\supp}{\operatorname{supp}}

\newcommand{\rng}{\operatorname{range}}
\newcommand{\spn}{\operatorname{span}}

\newcommand{\fbl}{\mathrm{FBL}}
\newcommand{\fbp}{{\mathrm{FBL}}^{(p)}}

\newcommand{\FVL}{\mathrm{FVL}}
\newcommand{\N}{{\mathbb N}}
\newcommand{\R}{{\mathbb R}}
\newcommand{\Z}{{\mathbb Z}}
\newcommand{\Q}{{\mathbb Q}}

\newcommand{\cC}{{\mathcal C}}
\newcommand{\cD}{{\mathcal D}}

\newcommand{\cS}{{\mathcal S}}

\newcommand{\cL}{{\mathcal L}}
\newcommand{\cU}{{\mathcal U}}

\newcommand{\cX}{{\mathcal X}}

\newcommand{\cP}{{\mathcal P}}
\newcommand{\so}{\operatorname{so}}
\newcommand{\co}{\operatorname{co}}

\title[Banach lattices with upper $p$-estimates]{Banach lattices with upper $p$-estimates: free and injective objects}

\author{E.~Garc\'ia-S\'anchez}
\address{Instituto de Ciencias Matem\'aticas (CSIC-UAM-UC3M-UCM)\\
Consejo Superior de Investigaciones Cient\'ificas\\
C/ Nicol\'as Cabrera, 13--15, Campus de Cantoblanco UAM\\
28049 Madrid, Spain.}
\email{enrique.garcia@icmat.es}

\author{D.~H.\ Leung}
\address{Department of Mathematics\\
National University of Singapore\\
Singapore 119076.
} \email{dennyhl@u.nus.edu}

\author{M.~A.\ Taylor}
\address{Department of Mathematics\\
ETH Z\"urich, Ramistrasse 101, 8092 Z\"urich, Switzerland.
} \email{mitchell.taylor@math.ethz.ch}

\author{P.~Tradacete}
\address{Instituto de Ciencias Matem\'aticas (CSIC-UAM-UC3M-UCM)\\
Consejo Superior de Investigaciones Cient\'ificas\\
C/ Nicol\'as Cabrera, 13--15, Campus de Cantoblanco UAM\\
28049 Madrid, Spain.}
\email{pedro.tradacete@icmat.es}

\date{\today}


\begin{document}
\begin{abstract}
We study the free Banach lattice $\fbl^{(p,\infty)}[E]$  with upper $p$-estimates generated by a Banach space $E$. Using a classical result of Pisier on factorization through $L^{p,\infty}(\mu)$ together with a  finite dimensional reduction, it is shown that the spaces $\ell^{p,\infty}(n)$ witness the universal property of $\fbl^{(p,\infty)}[E]$ isomorphically.  As a consequence, we obtain a functional representation for $\fbl^{(p,\infty)}[E]$,  answering a question from \cite{OTTT}.
More generally, our proof allows us to identify the norm of any free Banach lattice over $E$ associated with  a rearrangement invariant function space.

After obtaining the above  functional representation, we take the first steps towards analyzing the fine structure of $\fbl^{(p,\infty)}[E]$. Notably, we prove that the norm for $\fbl^{(p,\infty)}[E]$ cannot be isometrically witnessed by $L^{p,\infty}(\mu)$  and settle the question of characterizing when an embedding between Banach spaces extends to a lattice embedding between the corresponding free Banach lattices with upper $p$-estimates. To prove this latter result, we introduce a novel push-out argument, which when combined with the injectivity of $\ell^p$ allows us to give an alternative proof of the subspace problem for free $p$-convex Banach lattices. On the other hand, we prove that $\ell^{p,\infty}$ is not injective in the class of Banach lattices with upper $p$-estimates, elucidating one of many difficulties arising  in the study of  $\fbl^{(p,\infty)}[E]$. 
\end{abstract}

\maketitle
\allowdisplaybreaks
\tableofcontents
\section{Introduction}
 Given a Banach space $E$, the \emph{free Banach lattice generated by $E$} is a Banach lattice $\fbl[E]$ together with a linear isometric embedding $\phi_E: E\to \fbl[E]$ such that for every Banach lattice $X$ and every bounded linear operator $T:E\to X$ there is a unique lattice homomorphism $\widehat{T}:\fbl[E]\to X$ such that $\widehat{T}\circ \phi_E=T$ and $\|\widehat{T}\|=\|T\|.$ In other words, $\fbl[E]$ is the (unique) Banach lattice which converts linear operators defined on $E$ into lattice homomorphisms, in the sense of the following diagram:
\begin{center}
\tikzset{node distance=2cm, auto}

\begin{tikzpicture}
  \node (C) {$\fbl[E]$};
  \node (P) [below of=C] {$E$};
  \node (Ai) [right of=P] {$X$};
  \draw[->, dashed] (C) to node {$\exists ! \ \widehat{T}$} (Ai);
  \draw[<-] (C) to node [swap] {$\phi_E$} (P);
  \draw[->] (P) to node [swap] {$T$} (Ai);
\end{tikzpicture}

\end{center}
Over the past ten years, many novel techniques have been developed to understand the structural properties of $\fbl[E]$ as well as the rigidity of the correspondence $E\mapsto \fbl[E]$. This has led to a host of new developments in the theory of Banach spaces and Banach lattices, as well as to the resolution of several outstanding open problems. To indicate just a few examples, free Banach lattices are used in \cite{ART}  to solve a question of J.~Diestel on WCG spaces, in \cite{OTTT} to solve a question from \cite{TT}  regarding basic sequences satisfying maximal inequalities, and in \cite{Norm-attaining} to produce the first example of a lattice homomorphism which does not attain its norm. Free Banach lattices also play a fundamental role in the theory of projective Banach lattices \cite{AMR2,AMR,dePW} and are used in \cite{Aviles-Tradacete} to construct  Banach lattices of universal disposition, separably injective Banach lattices and push-outs. On the other hand, free Banach lattices  are known to have a remarkably rigid structure in their own right. Notably, it is proven in \cite{APR} that any disjoint collection in $\fbl[E]$ must be countable and in \cite{OTTT} that $\fbl[E]$ determines $E$ isometrically if $E^*$ is smooth.  This latter result has initiated a program to find geometric properties of a Banach space $E$ which can be equivalently characterized via lattice properties of $\fbl[E]$ -- see \cite{OTTT} for several such properties.
\medskip

Although  defined in an abstract manner, free Banach lattices admit a very concrete functional representation. To see this, denote by $H[E]$ the linear subspace of $\mathbb{R}^{E^*}$ consisting of all positively homogeneous functions $f:E^* \to \mathbb{R}$.
For $f\in H[E]$  define
\begin{equation*}\label{eq:ART}
 	\|f\|_{\fbl[E]}=
	\sup\left\{\sum_{k=1}^n |f(x_k^*)|: \, n\in\mathbb{N}, \, x_1^*,\dots,x_n^*\in E^*,
        \,  \sup_{x\in B_E} \sum_{k=1}^n |x_k^*(x)|\leq 1\right\}.
\end{equation*}
Given any $x\in E$, let $\delta_x\in H[E]$ be defined by
$$
	\delta_x(x^\ast):= x^\ast(x) 	\quad\mbox{for all }x^*\in E^*.
$$
It is easy  to see that
\begin{displaymath}
  H_1[E]:=\bigl\{f\in H[E] :  \|f\|_{\fbl[E]} <\infty\bigr\}
\end{displaymath}
is a sublattice of~$H[E]$ and that~$\|\cdot\|_{\fbl[E]}$ defines a
complete lattice norm on~$H_1[E]$. Moreover,
$\|\delta_x\|_{\fbl[E]}=\|x\|$ for every $x\in E$. As shown in \cite{ART},  $\fbl[E]$ coincides with the closed sublattice of $H_1[E]$ generated by $\{\delta_x:x\in E\}$ with respect to $\|\cdot\|_{\fbl[E]}$, together with the map $\phi_E(x)= \delta_x$. In particular, by positive homogeneity, we may view elements of $\fbl[E]$ as weak$^*$-continuous functions on the dual ball.
\medskip

For numerous reasons it is  useful to introduce a scale of free Banach lattices, indexed by $p\in [1,\infty]$. Recall that a Banach lattice $X$ is \emph{$p$-convex} if there exists a constant $M\geq 1$ such that for every $x_1,\ldots,x_m\in X$ we have
    \begin{equation}\label{P-convex def}
        \bigg\|\left(\sum_{k=1}^m|x_k|^p\right)^{\frac{1}{p}} \bigg\| \leq M \left(\sum_{k=1}^m\|x_k\|^p\right)^{\frac{1}{p}}.
    \end{equation}
 The least constant $M$ satisfying the above inequality is called the \emph{$p$-convexity constant of $X$} and is denoted by $M^{(p)}(X)$. We say that $X$ satisfies an \emph{upper $p$-estimate} (or that $X$ is a \emph{$(p,\infty)$-convex} Banach lattice) if \eqref{P-convex def} holds for all pairwise disjoint vectors $x_1,\ldots,x_m\in X$; the least constant is then called the \emph{upper $p$-estimate constant of $X$} and is denoted by $M^{(p,\infty)}(X)$. A classical result \cite[Proposition 1.f.6]{LT2} states that a Banach lattice $X$ has an upper $p$-estimate if and only if
\begin{equation}\label{upper p}
        \bigg\|\bigvee_{k=1}^m|x_k|\bigg\| \leq M\left(\sum_{k=1}^m\|x_k\|^p\right)^{\frac{1}{p}}
    \end{equation}
    for all $x_1,\dots, x_m\in X$. Moreover, the best choice of $M$ in \eqref{upper p} is exactly $M^{(p,\infty)}(X)$. By \cite{Byrd,LT2, Pisier},  every $p$-convex Banach lattice   (resp.~Banach lattice with an upper $p$-estimate) can be renormed to have $p$-convexity (resp.~upper $p$-estimate) constant one.   
\medskip

Given $p\in [1,\infty]$ and a Banach space $E$, one  defines the free $p$-convex Banach lattice generated by $E$ analogously to $\fbl[E]$: The \emph{free $p$-convex Banach lattice generated by $E$} is a Banach lattice $\fbp[E]$ with $p$-convexity constant one together with a linear isometric embedding $\phi_E: E\to \fbp[E]$ such that for every $p$-convex Banach lattice $X$ with $p$-convexity constant one and every bounded linear operator $T:E\to X$ there is a unique lattice homomorphism $\widehat{T}:\fbp[E]\to X$ such that $\widehat{T}\circ \phi_E=T$ and $\|\widehat{T}\|=\|T\|.$ By replacing $p$-convexity with upper $p$-estimates one obtains the definition of the \emph{free Banach lattice with upper $p$-estimates generated by $E$}, which is denoted by $\fbl^{(p,\infty)}[E]$.
\medskip

As shown in \cite{MR4326041},   $\fbp[E]$ admits an analogous functional representation  to  $\fbl[E]$, but with $\|\cdot\|_{\fbl[E]}$ replaced by
\begin{equation} \label{eq:ARTp}
 	\|f\|_{\fbp[E]}=
	\sup\left\{\left(\sum_{k=1}^n |f(x_k^\ast)|^p\right)^\frac{1}{p}: \, n\in\mathbb N, \, x_1^*,\dots,x_n^*\in E^*,
        \,  \sup_{x\in B_E} \left(\sum_{k=1}^n |x_k^\ast(x)|^p\right)^\frac{1}{p}\leq 1\right\}
\end{equation}
and $H_1[E]$ replaced by
\begin{displaymath}
  H_p[E]:=\bigl\{f\in H[E] :  \|f\|_{\fbp[E]} <\infty\bigr\}.
\end{displaymath}
Making  use of this explicit representation of $\fbp[E]$, most of the major results on $\fbl[E]$ can be shown  to hold for $\fbp[E]$. On the other hand, \cite[Sections 9.6 and 10]{OTTT} identify several interesting properties of $\fbp[E]$ which depend explicitly on $p$.
\medskip

The question of establishing a functional representation for $\fbl^{(p,\infty)}[E]$ was left as an open problem in \cite[Remark 6.2]{MR4326041} and then reiterated in \cite[Section 9.6]{OTTT}. The proof of the existence of  $\fbl^{(p,\infty)}[E]$ in \cite{MR4326041} proceeds by equipping the free vector lattice with a ``maximal" norm satisfying an upper $p$-estimate and then completing the resulting space.  However, this approach gives no practical information about the norm and cannot even guarantee that the resulting Banach lattice is a space of functions on the dual ball. For this reason, minimal progress has been made on $\fbl^{(p,\infty)}[E]$, with the exception of some interesting results being proven in \cite{PAP}.
\medskip

The first aim of this paper is to answer the above question and identify the norm for $\fbl^{(p,\infty)}[E]$. As a consequence, we  obtain the desired continuous injection $\fbl^{(p,\infty)}[E]\hookrightarrow C(B_{E^*})$ and the prospect of  ascertaining the fine structure of these spaces. As it turns out, up to equivalence, we have
\begin{equation} \label{eq:ARTup}
 	\|f\|_{\fbl^{(p,\infty)}[E]}=
	\sup\left\{\| (f(x_k^\ast))_{k=1}^n \|_{\ell^{p,\infty}(n)}: \, n\in\mathbb N, \, x_1^*,\dots,x_n^*\in E^*,
        \,  \sup_{x\in B_E}\| (x_k^\ast(x))_{k=1}^n\|_{\ell^{p,\infty}(n)}\leq 1\right\},
\end{equation}
where $\ell^{p,\infty}(n)$ denotes the canonical $n$-dimensional weak-$L^p$ space. Moreover, as we will see below, the proof that \eqref{eq:ARTup} is the correct norm is rather illuminating, in that it yields  a formula for the norm for the ``free Banach lattice" associated to other classes of lattices.
\medskip

From the above functional description of $\fbl^{(p,\infty)}[E]$,  we obtain immediate access to many results. For example, we  learn  from \cite{APR} that $\fbl^{(p,\infty)}[E]$ has only countable disjoint collections and from \cite{ART} that lattice homomorphic functionals separate the points of $\fbl^{(p,\infty)}[E]$. With more work, one can prove analogues of many of the results in \cite{OTTT} which were initially proven for $\fbp[E]$. However, certain deeper theorems do not immediately generalize to $\fbl^{(p,\infty)}[E]$. The purpose of the second half of this paper is to prove  results of this type. In particular, we solve the \emph{subspace problem} for $\fbl^{(p,\infty)}$ by characterizing when an embedding $\iota: F\hookrightarrow E$ induces a lattice embedding $\overline{\iota}: \fbl^{(p,\infty)}[F]\hookrightarrow \fbl^{(p,\infty)}[E]$. Obtaining such a characterization is by no means trivial, as the proof of the subspace problem for $\fbp$ relied heavily on the $L^p$-type structure of $p$-convex Banach lattices.
\subsection{Outline of the paper}\label{Intro sum}
In \Cref{s1} we study the free Banach lattice generated by $E$ associated to a non-empty class $\mathcal{C}$ of Banach lattices. We begin in \Cref{s1.1} by briefly reviewing the construction of the free vector lattice generated by $E$. After this, we equip the free vector lattice with the norm 
\[ \rho_\cC(f) := \sup\bigl\{\|\widehat{T}f\|_X: X \in \cC,\ T:E\to X \text{ is a linear contraction}\bigr\},\]
where $\wh T$ denotes the unique lattice homomorphic extension of $T$, and define the {\em free Banach lattice generated by $E$ associated to $\cC$} as   the completion of $(\FVL[E],\rho_\cC)$. This definition yields a Banach lattice $\fbl^\mathcal{C}[E]$ containing an isometric copy of $E$ with the ability to uniquely extend any operator $T:E\to X\in \mathcal{C}$ to a lattice homomorphism of the same norm. In \Cref{s2} we ``close" the class $\mathcal{C}$ by constructing a class $\overline{\mathcal{C}}\supseteq \mathcal{C}$ such that $\fbl^\mathcal{C}[E]=\fbl^\mathcal{\overline{C}}[E]\in \overline{\mathcal{C}}$. This then allows us to identify $\fbl^\mathcal{C}[E]$ as a universal object in this new class. In \Cref{s3} we  find an explicit formula for the norm of $\fbl^\mathcal{C}[E]$ when $\mathcal{C}$ consists of a single r.i.~space. This, in particular, allows us to identify  $\fbl^\mathcal{C}[E]$ as a space of weak$^*$-continuous positively homogeneous functions on $B_{E^*}.$ In \Cref{SEC 1.2} we characterize when continuous injective lattice homomorphisms extend injectively to the completion of a normed lattice, elucidating the difficulty in representing $\fbl^\mathcal{C}[E]$ inside of $C(B_{E^*})$.
\medskip

In \Cref{Application to convex} we  combine the results of \Cref{s1} with the Maurey and Pisier factorization theorems to reproduce the norm \eqref{eq:ARTp} for $\fbp[E]$ and show that \eqref{eq:ARTup} is a $(1-\frac{1}{p})^{\frac{1}{p}-1}$-equivalent lattice norm for $\fbl^{(p,\infty)}[E]$ with upper $p$-estimate constant one. We also characterize in \Cref{t3.6}  all of the  $\mathcal{C}$ for which we have $\fbl^\mathcal{C}[E]=\fbp[E]$ and all of the $\mathcal{C}'$ for which $\fbl^{\mathcal{C}'}[E]\approx \fbl^{(p,\infty)}[E]$. Then, in \Cref{Dist norm} we show that \eqref{eq:ARTup} is \emph{not} $1$-equivalent to the  norm satisfying the universal property of $\fbl^{(p,\infty)}[E]$. In fact, by characterizing the sequence spaces that $C$-lattice  embed into $(\bigoplus_{\mu\in \Gamma}L^{p,\infty}(\mu))_\infty$, we  are able to show that the entire class of weak-$L^p$ spaces (equipped with any of the renormings   in \Cref{lattice renorm}) fails to isometrically witness the  norm of the free Banach lattice with upper $p$-estimates.
\medskip

\Cref{Sub problem} is devoted to the subspace problem for  $\fbl^{(p,\infty)}$; that is, the problem of characterizing those embeddings $\iota: F\hookrightarrow E$ which induce a lattice embedding $\overline{\iota}: \fbl^{(p,\infty)}[F]\hookrightarrow \fbl^{(p,\infty)}[E]$. This is achieved in \Cref{s4.1} by utilizing a novel push-out argument. A benefit of our argument is that it applies equally well to $\fbp$. However, a priori, it yields a slightly different solution to the subspace problem for $\fbp$  than the one in \cite[Theorem 3.7]{OTTT}.  To reconcile this,  we prove  the injectivity of $\ell^p$ in the class of $p$-convex Banach lattices, which immediately yields the equivalence of the two solutions, up to constants. On the other hand, in \Cref{s4.3} we show that $\ell^{p,\infty}$ is \emph{not} injective in the class of Banach lattices with upper $p$-estimates, adding an additional novelty to our solution to the subspace problem for  $\fbl^{(p,\infty)}[E]$.

\section{The free Banach lattice associated to a class of Banach lattices}
\label{s1}
In this section we define and investigate the free Banach lattice $\fbl^\cC[E]$ generated by a Banach space $E$ associated to a class $\mathcal{C}$ of Banach lattices. We begin in \Cref{s1.1} with a brief review of the construction of the free vector lattice. Then, in \Cref{s2}, we define $\fbl^\cC[E]$ as the completion of $\FVL[E]$ under a certain norm,  identify an associated class $\overline{\mathcal{C}}\supseteq\mathcal{C}$ of Banach lattices, and  prove the universality of  $\fbl^\cC[E]$ in the class $\overline{\mathcal{C}}$. In \Cref{s3}, an explicit formula for the norm of $\fbl^\cC[E]$ is obtained when $\mathcal{C}$ consists of a single rearrangement invariant Banach lattice. This, in particular, yields a continuous injection $\fbl^\mathcal{C}[E]\hookrightarrow C(B_{E^*})$ for such $\mathcal{C}$. Finally, in \Cref{SEC 1.2} we supplement the above results with a characterization of the normed lattices for which every continuous lattice homomorphic injection extends injectively to the completion.

\subsection{Review of the free vector lattice}\label{s1.1}
Let $E$ be a Banach space and endow $B_{E^*}$ with the weak$^*$-topology.
For any $x\in E$, the function $\delta_x:B_{E^*}\to \R$ given by $\delta_x(x^*) = x^*(x)$ belongs to the space $C(B_{E^*})$ of continuous functions on $B_{E^*}$.
Define a sequence of subspaces of $C(B_{E^*})$ as follows:
\[ E_0= \{\delta_x: x\in E\};\ E_n = \spn(E_{n-1}\cup |E_{n-1}|) \text{  if $n\in \N$.}\]
Here, for a set $A$ in a vector lattice $X$, we are denoting $|A|=\{|x|:x\in A\}$. Let $\FVL[E]= \bigcup^\infty_{n=0}E_n$ and note that $E_0$ is a subspace of $C(B_{E^*})$ since $\sum^n_{k=1}a_k\delta_{x_k} = \delta_{\sum^n_{k=1}a_kx_k}$.

\begin{prop}\label{p1.1} $\FVL[E]$ is the smallest vector sublattice of $C(B_{E^*})$ containing $E_0$. \end{prop}

\begin{proof}
Obviously, $\spn(E_{n-1})\subseteq E_n$ for all $n\in \N$.  Hence, $\FVL[E]$ is a vector subspace of $C(B_{E^*})$ containing $\{\delta_x: x\in E\}$.
Assume that $f\in \FVL[E]$.  By definition, $f\in E_n$ for some $n\in \N$.
Thus, $|f| \in |E_n| \subseteq E_{n+1} \subseteq\FVL[E]$, so that $\FVL[E]$ is a vector sublattice of $C(B_{E^*})$.
\medskip

Assume that $G$ is a vector sublattice of $C(B_{E^*})$ containing $\{\delta_x: x\in E\}=E_0$.
If $E_{n-1}\subseteq G$ for some $n\in \N $, then $E_{n-1} \cup |E_{n-1}|\subseteq G$.
Hence, $E_n \subseteq G$.
By induction, $\FVL[E] = \bigcup^\infty_{n=0}E_n \subseteq G$.
\end{proof}




It follows from the expression for $\FVL[E]$ that each $f\in \FVL[E]$ is positively homogeneous in the following sense: $f(\lambda x^*) = \lambda f(x^*)$ if $x^*\in B_{E^*}$ and $0\leq \lambda\leq 1$.
Define $\vp := \vp_E:E \to \FVL[E]$ by $\vp(x) = \delta_x$. Clearly, $\vp$ is a linear operator.

\begin{prop}\label{p1.2} Let $f\in \FVL[E]$.  There is a finite subset $A$ of $E$
so that $f$ belongs to the sublattice of $\FVL[E]$ generated by $\{\delta_x: x\in A\}$.
\end{prop}

\begin{proof}The  assertion is trivial if $f\in E_0$.
Assume that it holds whenever $f\in E_{n-1}$.
Let $f\in E_n$ and represent $f = \sum^r_{k=1}(a_kf_k + b_k|f_k|)$ where $a_k, b_k\in \R$ and $f_k\in E_{n-1}$.  By the inductive hypothesis, there is a finite set $A$ so that $f_k, 1\leq k\leq r$,  belong  to the sublattice generated by $\{\delta_x:x\in A\}$. By construction, $f$ belongs to  the same sublattice.
%
\end{proof}

An alternative way to prove Propositions \ref{p1.1} and \ref{p1.2} is to utilize \cite[Exercise 8, p.~204]{AB}. The following statement is implicit in \cite[Theorem 2.1]{OTTT}, but we include the details for the convenience of the reader.

\begin{prop}\label{p1.3}
$(\FVL[E],\vp)$ is the free vector lattice generated by $E$ in the following sense: If $X$ is an Archimedean vector lattice and $T:E\to X$ is a linear operator then there is a unique linear vector lattice homomorphism $\wh{T}: \FVL[E] \to X$ such that $\wh{T}\vp = T$.
\end{prop}

\begin{proof}Let  $A$ be a finite set in $E$ and let  $L_A$ be the sublattice of $C(B_{E^*})$ generated by $\{\delta_x: x\in A\}$.
Let $I=I_A$ be the principal lattice ideal in $X$ generated by $\sum_{x\in A}|Tx|$.
By standard theory, there is  a compact Hausdorff space $K$ and an injective vector lattice homomorphism $i:I \to C(K)$.
For each $t\in K$, define $\ell_t: \text{span}\, A \to \R$ by $\ell_t(\sum_{x\in A}a_xx) = \sum_{x\in A}a_x(iTx)(t)$.
Note that if $\sum_{x\in A}a_xx=0$ then $\sum_{x\in A}a_xTx = 0$. This implies that $\sum_{x\in A}a_xiTx=0$ and thus $\ell_t(\sum_{x\in A}a_xx) = 0$.
Hence, $\ell_t$ is a well-defined linear functional on $\text{span}\, A$.  Since $\text{span}\, A$ is a finite dimensional subspace of $E$, $\ell_t$ is bounded.  Let $x_t^*\in E^*$ be a Hahn-Banach extension of $\ell_t$.
\medskip

Suppose that $f\in \FVL[E]$.  By Proposition \ref{p1.2}, there exists a finite set $A = \{x_1,\dots, x_n\}$ in $E$ so that $f\in L_A$.
We can therefore write $f = G(\delta_{x_1},\dots, \delta_{x_n})$ where $G$ is a finite sequence of operations of taking linear combinations and $|\cdot|$.
Define $\wh{T}f = G(Tx_1,\dots,Tx_n)\in X$.
To see that $\wh{T}$ is well-defined, we have to show that  $\wh{T}(\sum a_k f_k)=0$  if
$\sum a_kf_k=0$, for any finite sum.
We may assume that $f_k\in L_A$ for all $k$. In particular,  we have $\wh{T}f_k\in I_A$.
Write $f_k = G_k(\delta_{x_1},\dots, \delta_{x_n})$ for some lattice-linear $G_k$ and note that for any $t\in K$,
\begin{align*}  [i\wh{T}(\sum a_kf_k)](t) &= \sum a_k\, iG_k(Tx_1,\dots, Tx_n)(t)\\
& =\sum a_kG_k(iTx_1(t),\dots, iTx_n(t))\\
& = \sum a_k G_k(x^*_t(x_1),\dots, x^*_t(x_n))\\
& = \sum a_k G_k(\delta_{x_1}(x^*_t),\dots, \delta_{x_n}(x^*_t))\\
& = \sum a_kf_k(x^*_t).
\end{align*}
Thus, $\sum a_kf_k =0$ implies $i\wh{T}(\sum a_kf_k) =0$.
Since $i:I_A\to C(K)$ is injective, $\wh{T}(\sum a_kf_k) =0$.
This proves that $\wh{T}$ is well-defined.
It follows from its definition that $\wh{T}$ is a linear vector lattice homomorphism.
\medskip

It is easy to see that for any $x\in E$, $\wh{T}\vp x = \wh{T}(\delta_x) = Tx$.
Let $S:\FVL[E]\to X$ be a vector lattice homomorphism so that $S\vp x = Tx$. Then for any $f = G(\delta_{x_1},\dots, \delta_{x_n})$ we must have
\[ Sf = G(S\delta_{x_1}, \dots, S\delta_{x_n}) = G(Tx_1,\dots, Tx_n) = \wh{T}f.\]
This proves the uniqueness of $\wh{T}$.
\end{proof}
\begin{remark}
    In this paper we are identifying $\FVL[E]$ as a space of functions on $B_{E^*}$. This is in contrast to previous works where $\FVL[E]$ was identified as a space of functions on $E^*$. By positive homogeneity, these interpretations are clearly equivalent.
\end{remark}

\subsection{Universality of the free Banach lattice associated to a class of Banach lattices}
\label{s2}
Let $\cC$ be a (non-empty) class of Banach lattices. For instance, $\cC$  could be the class of weak-$L^p$ spaces or the class of Banach lattices satisfying an upper $p$-estimate with constant $1$.
Given a Banach space $E$, define a norm $\rho_\cC$ on $\FVL[E]$ by
\[ \rho_\cC(f) = \sup\bigl\{\|\wh{T}f\|_X: X \in \cC,\ T:E\to X \text{ is a linear contraction}\bigr\}.\]
Clearly, $\rho_\mathcal{C}$ is a lattice norm on $\FVL[E]$.
The {\em free Banach lattice generated by $E$ associated to $\cC$} is the completion of $(\FVL[E],\rho_\cC)$ and is denoted by $\fbl^\cC[E]$.  It is easy to see that $\vp_E: E\to \fbl^\cC[E]$ is a linear isometric embedding.
Define an enlarged class of Banach lattices $\ol{\cC}$ so that $Y\in \ol{\cC}$ if and only if $Y$ is  lattice isometric to a closed sublattice of a Banach lattice of the form $(\bigoplus_{\gamma\in \Gamma}X_\gamma)_\infty$, where $X_\gamma \in \cC$ for all $\gamma$. Note that we do not require that the $X_\gamma$'s  be  distinct.


\begin{prop}\label{p2.1}
Suppose that $Y\in \ol{\cC}$ and $T:E\to Y$ is a bounded linear operator.
Then $\wh{T}:(\FVL[E] ,\rho_\cC)\to Y$ is a  vector  lattice homomorphism such that $\|\wh{T}\| = \|T\|$.\end{prop}

\begin{proof}
Since $Y$ is an Archimedean vector lattice, $\wh{T}:\FVL[E]\to Y$ is a vector lattice homomorphism.
Note that $\vp_E$ is an isometric embedding and  $\wh{T}\vp_E = T$.  Hence, $\|T\| \leq \|\wh{T}\|$.
On the other hand, suppose that $Y$ is a closed sublattice of $X$, where $X = (\bigoplus_{\gamma\in \Gamma}X_\gamma)_\infty$, with $X_\gamma\in \cC$ for all $\gamma$.
Let $\pi_\gamma$ be the projection from $X$ onto $X_\gamma$.
Then $\pi_\gamma T:E\to X_\gamma$ is a bounded linear operator with $\|\pi_\gamma T\| \leq \|T\|$.  By definition of $\rho_\cC$, $ \|\widehat{\pi_\gamma T} f\| \leq \|T\|\, \rho_\cC(f)$ for any $f\in \FVL[E]$.
As $\pi_\gamma$ is a lattice homomorphism, by the uniqueness of $\widehat{\pi_\gamma T}$, we have $\widehat{\pi_\gamma T} = \pi_\gamma \wh{T}$.
Thus, $\|\wh{T}f\|  = \sup_\gamma\|\pi_\gamma \wh{T}f\| \leq\|T\|\, \rho_\cC(f)$ for any $f\in \FVL[E]$. This shows that $\|\wh{T}\|\leq \|T\|$.
\end{proof}

Obviously, the operator $\wh{T}$  in \Cref{p2.1} extends to a lattice homomorphism from $\fbl^\cC[E]$ to $Y$, still denoted by $\wh{T}$.

\begin{cor}\label{c2.2}
Let $E$ be a Banach space. For any non-empty class $\cC$ of Banach lattices, the norms $\rho_\cC$ and $\rho_{\ol{\cC}}$ agree on $\FVL[E]$.  In particular, $\fbl^\cC[E] = \fbl^{\ol{\cC}}[E]$ as Banach lattices.
\end{cor}

\begin{proof}
Since $\cC\subseteq \ol{\cC}$ it is clear that $\rho_\cC\leq \rho_{\ol{\cC}}$.
On the other hand, suppose that $f\in \FVL[E]$ and $\ep >0$ are given.
There exists $Y\in \ol{\cC}$ and  a linear contraction $T: E\to Y$ such that $\rho_{\ol{\cC}}(f) \leq (1+\ep) \|\wh{T}f\|$.
By Proposition \ref{p2.1}, $\|\wh{T}f\| \leq \|T\|\, \rho_\cC(f)= \rho_\cC(f)$.
It follows that $\rho_{\ol{\cC}} \leq \rho_\cC$.
\end{proof}

\begin{prop}\label{p2.2}
The Banach lattice  $\fbl^\cC[E]$ belongs to the class $\ol{\cC}$.
\end{prop}

\begin{proof}
For any $f\in \FVL[E]$ and $n\in \N$, there exists $X(f,n)\in \cC$ and a linear contraction $T = T_{f,n}:E\to X(f,n)$ so that $ \rho_\cC(f)\leq (1+ \frac{1}{n})\|\wh{T}f\|$. Let $X =  (\bigoplus_{f\in \FVL[E], n\in\N}X(f,n))_\infty$.  Clearly, $X\in \ol{\cC}$.
Define $j:\fbl^\cC[E]\to X$
by $jg = (\widehat{T_{f,n}}g)_{f,n}$.
Obviously, $j$ is a vector  lattice homomorphism. 
Let $g\in \FVL[E]$. By definition, $\rho_\cC(g) \geq \|\widehat{T_{f,n}}g\|$ for all $(f,n)$.  Hence, $j$ is a linear contraction.
On the other hand, $\|jg\| \geq \|\widehat{T_{g,n}}g\| \geq \frac{n}{n+1}\,\rho_\cC(g)$ for all $n$, implying that  $\|jg\| \geq \rho_\cC(g)$. This shows that $j:\FVL[E]\to X$ is a lattice isometry. Consequently, $j:\fbl^\cC[E]\to X$ is as well.
\end{proof}
The above discussion motivates the following question.
\begin{prob}\label{pr2.4}
    Find necessary and sufficient conditions on classes of Banach lattices $\cC$ and $\cD$ so that $\rho_\cC = \rho_\cD$, or that $\rho_\cC$ is equivalent to $\rho_\cD$.
\end{prob}
In \Cref{Solution to C=Cbar} we will provide a complete solution to  \Cref{pr2.4} when $\cC$ is the class of $p$-convex Banach lattices and when $\cC$ is the class of Banach lattices with upper $p$-estimates.

\subsection{Rearrangement invariant spaces and the norm for \texorpdfstring{$\fbl^\cC[E]$}{}}\label{s3}
In this subsection, let $(\Om, \Sigma,\mu)$ be a non-atomic $\sigma$-finite measure space and let $X$ be a rearrangement invariant (r.i.) Banach function space on $\Om$ in the sense of \cite[Definition II.4.1]{BennettSharpley}.
To be specific, $X$ is a vector lattice ideal of $L^0(\Om,\Sigma,\mu)$, equipped with a complete lattice norm $\|\cdot\|$, so that $f \in X$, $g \in L^0(\Om,\Sigma,\mu)$ and $\mu\{|f| > t\} = \mu\{|g| > t\}$ for all $t \geq 0$ imply that $g\in X$ and $\|f\| = \|g\|$.
We further assume that $X$ has the Fatou property: Whenever $f_n\in X$, $0\leq f_n \uparrow f$ a.e.\ and $\sup\|f_n\| <\infty$ then $f\in X$ and $\|f\| = \sup_n\|f_n\|$. 
\medskip

Let $\cC$ be the class consisting of $X$ only.
Given a Banach space $E$, the aim of this subsection will be to produce a formula for the norm $\rho_\cC$ on $\FVL[E]$ and show that $\fbl^\cC[E] \subseteq C(B_{E^*})$.
\medskip

Given a sequence $\cU = (U_n)$  of disjoint measurable subsets of $\Om$ so that $0<\mu(U_n) < \infty$ for all $n$, we  define for each $f\in L^0(\Om,\Sigma,\mu)$ the conditional expectation over $\cU$ by
\[P_\cU f = \sum_n \frac{\int_{U_n}f\,d\mu}{\mu(U_n)}\chi_{U_n}.\]
Note that the above sum has at most one non-zero term at any given point of $\Omega$, so is well-defined as long as $\int_{U_n}f\,d\mu$ is finite for every $n\in \N$. This will be the case when $f\in X$.

\begin{prop}
Let $\cU = (U_n)$ be a sequence of disjoint measurable subsets of $\Om$ so that $\mu(U_n) < \infty$ for all $n$.
Then $P_\cU f\in X$ if $f\in X$ and $P_\cU$ is a positive contraction on $X$.
\end{prop}

\begin{proof}
    First note that, since $X$ is a Banach function space, $\int_{U_n}f\,d\mu$ is finite for every $n\in\N$ (see property 5 in \cite[Definition I.1.1]{BennettSharpley}), so $P_\cU f$ is well-defined. It is straightforward to check that $P_\cU$ is a positive contraction from $L^1(\Om,\Sigma,\mu)$ to $L^1(\Om,\Sigma,\mu)$ and from $L^\infty(\Om,\Sigma,\mu)$ to $L^\infty(\Om,\Sigma,\mu)$. Thus, $P_\cU$ is an admissible operator for the compatible pair $(L^1(\Om,\Sigma,\mu), L^\infty(\Om,\Sigma,\mu))$. Moreover, since $X$ is a r.i.~space over a non-atomic $\sigma$-finite  (in particular, resonant) measure space, it follows from \cite[Theorem II.2.2]{BennettSharpley} that $X$ is an exact interpolation space, so $P_\cU:X\rightarrow X$ is a positive contraction.
\end{proof}

\begin{lem}\label{l2.4}
Let $h_1,\dots, h_n\in X$ and set $h = \sum^n_{k=1}|h_k|$.  Given $\ep > 0$,
there is a sequence  $\cU = (U_n)$ of disjoint measurable subsets of $\Om$  so that $0<\mu(U_n) < \infty$ for all $n$ and
$|h_k - \cP_\cU h_k| \leq \ep h$, $1\leq k\leq n$.
\end{lem}

\begin{proof}
We may assume that $\ep <1$.
Let $V_i = \{\ep^{i} \leq h < \ep^{i-1}\}$.  Then $(V_i)_{i\in \Z}$ is a sequence of disjoint measurable subsets of $\Om$ so that $h = h_k = 0$ on $(\bigcup_i V_i)^c$ for any $k$.
Since the measure space is $\sigma$-finite, for each $i$ there is an at most countable measurable partition $(U_{ij})_{j\in J_i}$ of $V_i$ with $0<\mu(U_{ij})<\infty$ and $b_{kij} \in \R$ so that
\[ |h_k\chi_{V_i} - \sum_{j\in J_i} b_{kij}\chi_{U_{ij}}| \leq \frac{\ep^{i+1}}{2},\  i\in \Z,\  1\leq k\leq n.\]
Thus, for each $1\leq k\leq n$ we have
\begin{equation}\label{e1.1} |h_k - \sum_{i,j} b_{kij}\chi_{U_{ij}}| \leq \sum_i \frac{\ep^{i+1}}{2}\chi_{V_i}\leq \frac{\ep h}{2}
\end{equation}
and
\[  \bigl|\frac{1}{\mu(U_{ij})}\int_{U_{ij}}h_k  \,d\mu - b_{kij} \bigr| = \bigl|\frac{1}{\mu(U_{ij})}\int_{U_{ij}}h_k - b_{kij}\chi_{U_{ij}}\,d\mu\bigr| \leq \frac{\ep^{i+1}}{2}.
\]
Take $\cU = (U_{ij})_{i\in \Z, j\in J_i}$. Note that $\cP_\cU h_k = 0$ on the set $(\bigcup_{ij}U_{ij})^c$.
For $1\leq k\leq n$,
\begin{equation}\label{e1.2} |\cP_\cU h_k - \sum_{i,j}b_{kij}\chi_{U_{ij}}| \leq \sum_{i,j}\frac{\ep^{i+1}}{2}\chi_{U_{ij}} = \sum_i \frac{\ep^{i+1}}{2}\chi_{V_i}\leq \frac{\ep h}{2}.
\end{equation}
Summing (\ref{e1.1}) and (\ref{e1.2}) gives the desired result.
\end{proof}

Let $I = \{(m,r) \in \N^2: \frac{m}{r} \leq \mu(\Om)\}$.
If $(m,r)\in I$, take any  sequence $(V_i)^m_{i=1}$ of disjoint measurable subsets of $\Om$ such that $\mu(V_i) = r^{-1}$, $1\leq i\leq m$.
Define a norm $\rho_{mr}$ on $\R^m$  by
\[ \rho_{mr}(a_1,\dots,a_m) = \|\sum^m_{i=1}a_i\chi_{V_i}\|_X.\]
Let $E$ be a Banach space and define $\rho :C(B_{E^*})\to [0,\infty]$   by
\[ \rho(f) = \sup \rho_{mr}(f(x^*_1),\dots, f(x^*_m)),\]
where the supremum is taken over all $(m,r)\in I$ and $x^*_1,\dots, x^*_m\in E^*$ which satisfy the constraint $\sup_{x\in B_E}\rho_{mr}(x_1^*(x),\dots, x_m^*(x)) \leq 1$.
By definition, $\rho(\delta_x) = \|x\|$ for any $x\in E$.
Moreover, it is easy to check that $\rho$ is a complete lattice norm on the vector lattice $\{f\in C(B_{E^*}): \rho(f) < \infty\}$ which contains $\FVL[E]$. For our next result, recall that we are considering  the case $\cC=\{X\}$.

\begin{thm}\label{t2.5}
The norms $\rho$ and $\rho_\cC$ are equal on $\FVL[E]$. In particular, $\fbl^\cC[E] \subseteq \{f\in C(B_{E^*}): \rho(f) < \infty\}$.
\end{thm}

\begin{proof}
Suppose that $f\in \FVL[E]$. There are $x_1,\dots, x_n\in E$ and a lattice-linear function $G$ so that $f = G(\delta_{x_1},\dots, \delta_{x_n})$.
Let $\ep >0$ be given.
There are $(m,r)\in I$ and $x^*_1,\dots, x^*_m\in E^*$ so that
$\sup_{x\in B_E}\rho_{mr}(x_1^*(x),\dots, x_m^*(x)) \leq 1$ and
\[ \rho(f) \leq \rho_{mr}(f(x^*_1),\dots, f(x^*_m)) + \ep = \|\sum^m_{i=1}f(x^*_i)\chi_{V_i}\|_X + \ep,\]
where $(V_i)^m_{i=1}$ is the sequence of disjoint measurable subsets of $\Om$ chosen above, so that $\mu(V_i) = r^{-1}$, $1\leq i\leq m$.
Define a linear  operator $S:\FVL[E]\to X$
by $S g = \sum^m_{i=1}g(x^*_i)\chi_{V_i}$.
It is clear that $S$ is a lattice homomorphism.
By definition of $\rho$, if $\rho(g) \leq 1$, then $\rho_{mr}(g(x^*_1),\dots, g(x^*_m)) \leq 1$
and thus $\|Sg\|\leq 1$.
Let $T = S\vp_E:E \to X$.  Then $\|T\| \leq 1$ and hence
\[ \rho_\cC(f) \geq \|\wh{T}f\| = \|\widehat{S\vp_E}f\| = \|G(S\delta_{x_1},\dots,S\delta_{x_n})\| = \|Sf\|,\]
where the last step holds since $S$ is a lattice homomorphism.
Thus, $\rho_\cC(f) \geq \rho(f) -\ep$ for any $\ep >0$. Therefore, $\rho_\cC(f) \geq \rho(f)$.
\medskip

Conversely, given $\ep > 0$, there exists a linear contraction $T:E\to X$ so that $\rho_\cC(f) \leq \|\wh{T}f\| + \ep$.
Let  $h_k =  Tx_k$, $1\leq k\leq n$, and $h = \sum^n_{k=1}|h_k|$.
By Lemma \ref{l2.4}, there is a sequence $\cU = (U_i)$ of disjoint measurable subsets of $\Om$ so that $0<\mu(U_i)<\infty$ for all $i$ and  $|h_k-\cP_\cU h_k| \leq \ep h$, $1\leq k \leq n$.
There is a constant $C_G < \infty$, depending only on $G$, so that
\[ |\wh{T}f - G(\cP_\cU h_1,\dots, \cP_\cU h_n)| = |G(h_1,\dots, h_n) - G(\cP_\cU h_1,\dots, \cP_\cU h_n)| \leq C_G\ep h.\]
Express $\cP_\cU h_k$ as $\sum_i a_{ki}\chi_{U_i}$, where $T^*\chi_{U_i} = y^*_i$ and
\[ a_{ki} = \frac{1}{\mu(U_i)}\int_{U_i}h_k\,d\mu =\frac{y^*_i(x_k)}{\mu(U_i)}.\]
 Then
\begin{align*}
G(\cP_\cU h_1,\dots, \cP_\cU h_n)  &= G(\sum_ia_{1i}\chi_{U_i},\dots, \sum_ia_{ni}\chi_{U_i})\\
& = \sum_i G(a_{1i},\dots,a_{ni})\chi_{U_i}
\\
 & = \sum_i G(\delta_{x_1},\dots, \delta_{x_n})(y^*_i)\,\frac{\chi_{U_i}}{\mu(U_i)}
 \\& = \sum_i \frac{f(y^*_i)}{\mu(U_i)}\,\chi_{U_i}.
\end{align*}
Since $X$ has the Fatou property, there exists $l\in \N$ so that
\[ \|\sum^l_{i=1} \frac{f(y^*_i)}{\mu(U_i)}\,\chi_{U_i}\| > \|G(\cP_\cU h_1,\dots, \cP_\cU h_n)\| -\ep \geq \|\wh{T}f\| - (1+ C_G\|h\|)\ep.\]
Using the Fatou property once more,
choose disjoint Lebesgue measurable sets $V_1,\dots, V_l$ so that $V_i \subseteq U_i$, $\mu(V_i) \in \Q$ and
  \[ \|\wh{T}f\| - (1+ C_G\|h\|)\ep< \|\sum^l_{j=1}\frac{f(y^*_i)}{\mu(U_i)}\,\chi_{V_i}\|.\]
Write $\mu(V_i) =\frac{j_i}{r}$, where $j_i,  r\in \N$. 
Decompose each $V_i$ as a disjoint union $V_i = \bigcup^{j_i}_{s=1}V_{is}$, with $\mu(V_{is}) = r^{-1}$.
Take $j_0 = \sum^l_{i=1}j_s$.  It follows that $(j_0,r) \in I$.
Consider the sequence where each $\frac{y^*_i}{\mu(U_i)}$ is repeated $j_i$ times.  For any $x\in B_E$, let $b$ be the element of $\R^{j_0}$ so that each $\frac{y^*_i(x)}{\mu(U_i)}$ occurs $j_i$ times.
Then
\[ \rho_{j_0r}(b) = \|\sum^l_{i=1}\frac{y^*_i(x)}{\mu(U_i)}\,\chi_{V_i}\| \leq  \|\cP_\cU(Tx)\| \leq \|T\| =1.
\]
Recall that elements of $\FVL[E]$ have a homogeneity property given after Proposition \ref{p1.1}.
It follows from the definition of $\rho$ that, taking  $a$ to be the element of $\R^{j_0}$ where each $\frac{f(y^*_i)}{\mu(U_i)}$ occurs $j_i$ times,
\begin{align*}
\rho(f)  &\geq \rho_{j_0r}(a) =  \|\sum^l_{i=1}\sum^{j_i}_{s=1} \frac{f(y^*_i)}{\mu(U_i)}\,\chi_{V_{is}}\|
 = \|\sum^l_{i=1}\frac{f(y^*_i)}{\mu(U_i)}\,\chi_{V_i}\|\\
 &> \|\wh{T}f\| - (1+ C_G\|h\|)\ep\geq \rho_\cC(f) - (2+ C_G\|h\|)\ep.
\end{align*}
Thus, $\rho(f) \geq \rho_\cC(f)$.
\end{proof}
The above construction of $\fbl^\mathcal{C}[E]$ motivates the following question.
\begin{question}
    It is well-known that the properties of $\fbp[E]$ depend heavily on whether $p$ is finite or infinite.  Can the Banach lattices $\fbl^\mathcal{C}[E]$ be used to interpolate such properties? For example, if $(e_k)$ is the unit vector basis of $\ell^2,$ then by \cite[Propositions 5.14 and 6.4]{OTTT} $(|\delta_{e_k}|)$ gives a copy of $\ell^1$ in $\fbp[\ell^2]$ when $p<\infty$ and a copy of $\ell^2$ when $p=\infty$. Can we find for each $r\in (1,2)$ a class $\mathcal{C}_r$ so that $(|\delta_{e_k}|)$ behaves like $\ell^r$ in $\fbl^{\mathcal{C}_r}[\ell^2]$?
\end{question}

\subsection{Injective homomorphisms need not extend to the norm completion}\label{SEC 1.2}
In  \Cref{s2} we defined $\fbl^\mathcal{C}[E]$ as the completion of $(\FVL[E],\rho_\mathcal{C})$. Clearly, there is a continuous lattice homomorphic injection $(\FVL[E],\rho_\mathcal{C})\hookrightarrow C(B_{E^*})$ because of the inequality $\|\cdot\|_\infty\leq \rho_\mathcal{C}(\cdot)$ on $\FVL[E]$. In \cite{MR1664353} it is shown that under certain order continuity assumptions (which are not fulfilled in our situation) continuous lattice homomorphic injections  retain their injectivity after norm completion. Moreover, it is claimed in \cite[Lemma 14]{Schiavo} that \emph{any} continuous lattice homomorphic injection from a normed vector lattice to a Banach lattice lifts injectively to the completion. If this were the case, it would provide a straightforward solution to the question of finding a function lattice representation for $\fbl^{(p,\infty)}[E]$  asked in \cite[Remark 6.2]{MR4326041}. In this subsection, we give a complete characterization of when  continuous injective lattice homomorphisms  extend injectively to the completion, and, in particular, show that this need not always be the case.
\medskip

 Throughout this subsection  $X$ will be a vector lattice equipped with a lattice norm and $Y$ will be a Banach lattice.  We denote the norm completion of $X$ by $\wh{X}$ and consider polar sets with respect to the duality $\la \wh{X},X^*\ra$: for $A\subseteq \wh{X}$, set $A^\circ=\{x^*\in X^*:\langle a,x^*\rangle\leq1,\,\forall a\in A\}.$

\begin{thm}\label{t3}
Let $X$ be a normed vector lattice equipped with a lattice norm and let $\wh{X}$ be its norm completion.  The following are equivalent.
\begin{enumerate}
\item For any Banach lattice $Y$, any bounded linear injective lattice homomorphism $T:X\to Y$ extends to a bounded linear injective operator $\wh{T}:\wh{X}\to Y$.
\item For any $0 < x \in \wh{X}$, $I_x\cap X\neq \{0\}$, where $I_x$ is the closed ideal in $\wh{X}$ generated by $x$.
\end{enumerate}
\end{thm}

\begin{proof}
(1)$\implies$(2):  Suppose that there exists $0<x\in\wh{X}$ so that $I_x\cap X = \{0\}$.
By \cite[Proposition II.4.7]{Schaefer74}, $(I_x)^\circ$ is an ideal in $X^*$.
Define $\rho:X\to \R$ by
\[ \rho(u) = \sup_{x^*\in (I_x)^\circ\cap B_{X^*}}|\la u, x^*\ra|.\]
Since $(I_x)^\circ \cap B_{X^*}$ is a solid subset of $B_{X^*}$, $\rho$ is a lattice semi-norm on $X$ such that $\rho(\cdot) \leq \|\cdot\|$.
If $0 < u\in X$, then $u\notin I_x$.  Since $I_x$ is a closed subspace of $\wh{X}$, there exists $x^*\in (I_x)^\circ \cap B_{X^*}$ such that $\la u,x^*\ra\neq 0$.
Hence, $\rho(u) >0$.  This shows that $\rho$ is a norm on $X$.
\medskip

Denote by $X_\rho$ the vector lattice $X$ normed by $\rho$ and let $i:X\to X_\rho\subseteq \wh{X_\rho}$ be the formal identity, where $\wh{X_\rho}$ is the norm completion of $X_\rho$.
By construction, $i$ is a bounded linear injective lattice homomorphism.  By (1), $i$ extends to an injective bounded linear operator $\wh{i}:\wh{X}\to \wh{X_\rho}$.
Denote by $\wh{\rho}$ the norm on $\wh{X_\rho}$.
By injectivity of $\wh{i}$, $\wh{i}x\neq 0$.  Let $c = \wh{\rho}(\wh{i}x) > 0$.
There exists $(x_n)$ in $X$ that converges to $x$ in $\wh{X}$. By continuity, $(ix_n)$ converges in $\wh{X_\rho}$ to $\wh{i}x$.
Choose $N$ large enough so that $\wh{\rho}(ix_N-\wh{i}x) < \frac{c}{2}$.
Then $\rho(ix_N) > \frac{c}{2}$.  By definition of $\rho$, there exists $x^*\in (I_x)^\circ \cap B_{X^*}$ so that
$|\la ix_N, x^*\ra| > \frac{c}{2}$.
By definition of $\rho$ again, $x^*$ defines a bounded linear functional on $X_\rho$ with norm $\rho^*(x^*)\leq 1$.  
Hence,
\[  |\la x,x^*\ra| = |\la \wh{i}x,x^*\ra| \geq |\la ix_N,x^*\ra| - \wh{\rho}(ix_N-\wh{i}x)\cdot \rho^*(x^*) > \frac{c}{2} - \frac{c}{2} = 0.
\]
This is impossible since $x\in I_x$ and $x^*\in (I_x)^\circ$.

\medskip

\noindent(2)$\implies$(1):  Let $T:X\to Y$ be as in (1), and assume that its extension $\wh{T}$ is not injective. There exists some $x\in \wh{X}$, $x\neq 0$, such that $\wh{T}x=0$. Since $\wh{T}$ is a lattice homomorphism, we can assume that $x>0$. Let $I_x$ be the closed ideal in $\wh{X}$ generated by $x$. By (2), there is a non-zero element $z\in I_x\cap X$. Since $I_x$ is the closure of the ideal generated by $x$, there exists a sequence $(z_n)\subseteq \wh{X}$ converging to $z$ in $\wh{X}$ and some positive scalars $(\lambda_n)$ such that $0\leq |z_n|\leq \lambda_n x$. Therefore,
\[0\leq |\wh{T}z_n|\leq \lambda_n \wh{T}x=0 \]
and
\[Tz=\wh{T}z = \lim_n \wh{T}z_n=0, \]
so by the injectivity of $T$ we conclude that $z$ must be zero. This is a contradiction.
\end{proof}

The following example taken from \cite{MR4626544} illustrates the above situation.

\begin{ex}
There is a dense vector sublattice  $X$ of $c_0$ such that $I\cap X =\{0\}$, where $I$ is the closed ideal in $c_0$ generated by $e_1$; that is, $I = \{(a_n)\in c_0: a_n =0 \text{ for all $n\geq 2$}.\}$. Consequently, not every bounded linear injective lattice homomorphism from $X$ into a Banach lattice $Y$ extends to an injective operator from $c_0$ to $Y$.
\end{ex}

\begin{proof}
Let $X$ consist of all $(a_n)\in c_0$ such that there exists $m\in \N$ so that $a_n = \frac{a_1}{n}$ for all $n \geq m$.
Clearly, $X$ is a vector sublattice of $c_0$ with $e_n \in X$ if $n\geq 2$.
For any $n\geq 2$, $x_n: = e_1 + \sum^\infty_{k=n}\frac{e_k}{k} \in X$ and $(x_n)$ converges to $e_1$ in $c_0$.  Thus, $e_1\in \ol{X}$.  It follows that $\ol{X} = c_0$.
If $a = (a_n) \in I\cap X$, then $a_n= 0$ for all $n\geq 2$.  Since $a_1 = ma_m$ for sufficiently large $m$, $a_1=0$ as well.  Thus, $a =0$.
\end{proof}

\section{Applications to free \texorpdfstring{$p$}{}-convex and free \texorpdfstring{$(p,\infty)$}{}-convex Banach lattices}\label{Application to convex}

We now apply the results of Sections \ref{s2} and \ref{s3} to the classes $\cC_p$ and $\cC_{p,\infty}$ consisting   of all Banach lattices with  $p$-convexity  constant $1$ and  all  Banach lattices with upper $p$-estimate constant $1$,  respectively. For this purpose, we recall the following factorization theorems.

\begin{thm}[Maurey's Factorization Theorem]\label{t3.1} Let  $1 < p <\infty$, $A \subseteq L^1(\mu)$ and $0 <  M < \infty$. The following are equivalent.
\begin{enumerate}
\item For all finitely supported sequences $(\al_i)_{i\in I}$ of real numbers and $(f_i)_{i\in I}\subseteq A$,
\[ \|(\sum |\al_if_i|^p)^{\frac{1}{p}}\|_{L^{1} (\mu)} \leq M(\sum |\al_i|^p)^{\frac{1}{p}}.\]
\item  There exists $g\in L^1(\mu)_+$, $\|g\|_1= 1$, such that for any $f\in A$,
\[ \bigl\|\frac{f}{g}\bigr\|_{L^p(g\cdot \mu)} \leq M.\]
\end{enumerate}
\end{thm}

\begin{thm}[Pisier's Factorization Theorem] \label{t3.2} Let $1 < p <\infty$, $A \subseteq L^1(\mu)$, $0 <  M < \infty$ and $\gamma_p = (1-\frac{1}{p})^{\frac{1}{p}-1}$. Consider the following conditions.
\begin{enumerate}
\item For all finitely supported sequences $(\al_i)_{i\in I}$ of real numbers and $(f_i)_{i\in I}\subseteq A$,
\[ \bigl\|\bigvee |\al_if_i|\bigr\|_{L^{1} (\mu)} \leq M(\sum |\al_i|^p)^{\frac{1}{p}}.\]
\item There exists $g\in L^1(\mu)_+$, $\|g\|_1\leq 1$, such that for any $f\in A$ and any $\mu$-measurable set $U$,
\[ \bigl\|f\chi_U\bigr\|_{L^{1} (\mu)} \leq \gamma_p M\,(\int_Ug\,d\mu)^{1-\frac{1}{p}}.\]
\end{enumerate}
Then (1)$\implies$(2).
\end{thm}
\Cref{t3.1} is very well-known. See, for example, \cite[Chapter 7]{AK}. \Cref{t3.2}  is the implication (iii)$\implies$(i) in \cite[Theorem 1.1]{Pisier}. Note that condition (2) in \Cref{t3.2} is equivalent to
\[ \bigl\|\frac{f}{g}\chi_U\|_{L^1 (g\cdot \mu)} \leq \gamma_p M\bigl((g\cdot\mu)(U)\bigr)^{1-\frac{1}{p}}.
\]
In particular, under the above conditions, $\bigl\|\frac{f}{g}\bigr\|_{L^{p,\infty} (g\cdot \mu)} \leq \gamma_p M$ for all $f\in A$.
\medskip

We will use the above factorization theorems to represent $p$-convex Banach lattices inside of infinity sums of $L^p$ spaces and Banach lattices with upper $p$-estimates inside of infinity sums of weak-$L^p$ spaces. Recall that, given a measure space $(\Omega, \Sigma,\mu)$, the space $L^{p,\infty}(\mu)$ is defined as the set of all measurable functions $h:\Omega\rightarrow \R$ such that the quasinorm
\[  \sup_{t>0} t \mu(\{|h|>t\})^{\frac{1}{p}}   \]
is finite. Note, in particular, that every $h\in  L^{p,\infty}(\mu)$ has $\sigma$-finite support. Unless otherwise specified, we will equip $L^{p,\infty}(\mu)$ with the norm
\[ \|h\|_{L^{p,\infty}} = \sup \{\mu(E)^{\frac{1}{p}-1}\int_E |h| d\mu: 0 < \mu(E) <\infty\}.\]
It is  a standard fact (cf.~\cite[Exercise 1.1.12]{GrafakosCFA}) that the above norm is equivalent to the usual quasinorm on weak-$L^p$; it can also be checked to be a lattice norm with upper-$p$ estimate constant one.
\medskip


\begin{prop}\label{p3.3}The following statements hold:
\begin{enumerate}
    \item Let $X$ be a $p$-convex Banach lattice with constant $M$.
There is a family $\Gamma$ of probability measures and a lattice isomorphism
\[
J : X\to \bigl(\oplus_{\mu\in\Gamma}L^p(\mu)\bigr)_{\infty}\]
such that $\|x\| \leq \|Jx\| \leq M\|x\|$ for all $x\in X$.
    \item Let $X$ be a $(p,\infty)$-convex Banach lattice with constant $M$.
There is a family $\Gamma$ of probability measures and a lattice isomorphism
\[ J : X\to \bigl(\oplus_{\mu\in \Gamma}L^{p,\infty}(\mu)\bigr)_{\infty}\]
such that $\|x\| \leq \|Jx\| \leq \gamma_p M\|x\|$ for all $x\in X$.
\end{enumerate}
\end{prop}

\begin{proof}
Let $x\in S_X^+=\{z\in X_+:\|z\|=1\}$. Choose $x^*\in X^*_+$, $\|x^*\|=1,$ such that $x^*(x) = 1$.
Define $\rho_x:X\to \R$ by $\rho_x(z) = x^*(|z|)$.  Then $\rho_x$ is an $L$-norm on $X/\ker \rho_x$. Let $q_x: X\to X/\ker\rho_x$ be the quotient map.
There exist a probability measure $\mu_x$ and a contractive lattice homomorphism $i_x: X/\ker\rho_x\to L^1(\mu_x)$ such that $i_xq_xx = 1$, the constant $1$ function. 
Set $A_x = i_xq_x(B_X)\subseteq L^1(\mu_x)$.
Let $\al_i$ be a finitely supported real sequence and let $f_i = i_xq_xx_i$ for $x_i \in B_X$.
Since $i_xq_x$ is a lattice homomorphism,
\begin{align*}
\|(\sum |\al_if_i|^p)^{\frac{1}{p}}\|_{L^1(\mu_x)} & = \la (\sum |\al_ix_i|^p)^{\frac{1}{p}}, x^*\ra\leq \|(\sum |\al_ix_i|^p)^{\frac{1}{p}}\|\\
&\leq M(\sum|\al_i|^p)^{\frac{1}{p}} \text{ if $X$ is $p$-convex with constant $M$},
\\ \|\bigvee |\al_if_i|\|_{L^1(\mu_x)} & = \la \bigvee |\al_ix_i|, x^*\ra\leq \|\bigvee |\al_ix_i|\|\\
&\leq M(\sum|\al_i|^p)^{\frac{1}{p}} \text{ if $X$ is $(p,\infty)$-convex with constant $M$}.
\end{align*}
By the factorization theorems above, there exists $g_x\in L^1(\mu_x)_+$, $\|g_x\|_{L^1(\mu_x)}\leq 1$, such that for any $f\in A$ and any $\mu_x$ measurable set $U$,
\[ \bigl\|\frac{f}{g_x}\bigr\|_{L^p(g_x\cdot \mu_x)} \leq M,\ \text{respectively},\
\bigl\|\frac{f}{g_x} \bigr\|_{L^{p,\infty} (g_x\cdot \mu_x)} \leq \gamma_p M.\]
Define \[
J : X\to \bigl(\oplus_{x\in S_{X}^+}L^p(g_x\cdot\mu_x)\bigr)_{\infty}, \quad \text{respectively}\quad
J : X\to \bigl(\oplus_{x\in S_{X}^+}L^{p,\infty}(g_x\cdot\mu_x)\bigr)_{\infty}\]
by $Jz = (i_xq_x z/g_x)_x$.  It is easy to check that $J$ is a lattice homomorphism.
If $z\in B_X$, we have  $i_xq_xz\in A_x$.  Hence,
\[ \|i_xq_xz/g_x\|_{L^p(g_x\cdot\mu_x)} \leq M,\ \text{respectively}\  \|i_xq_xz/g_x\|_{L^{p,\infty}(g_x\cdot\mu_x)} \leq \gamma_p M.
\]
Thus, $\|Jz\| \leq M\|z\|$, respectively $\|Jz\| \leq \gamma_p M\|z\|$ for all $z\in X$.
\medskip

On the other hand, if $x\in S_X^+$, then in the $p$-convex case,
\[ \|Jx\| \geq  \|i_xq_xx/g_x\|_{L^p(g_x\cdot\mu_x)} \geq \|i_xq_xx/g_x\|_{L^1(g_x\cdot\mu_x)} \]
since $\|g_x\cdot \mu_x\|\leq 1$.
Therefore, $\|Jx\| \geq \|i_xq_xx\|_{L^1( \mu_x)}= x^*(x) =1$.
In the $(p,\infty)$-convex case, suppose that $\mu_x$ is a measure defined on $\Om_x$. Then
\[  \int  \frac{i_xq_xx}{g_x}\,d(g_x\cdot \mu_x) = \int  {i_xq_xx}\,d\mu_x= x^*(x) =1 \geq  \bigl((g_x\cdot \mu_x)(\Om_x)\bigr)^{1-\frac{1}{p}}.
\]
Hence, $\|Jx\|\geq \bigl\|\frac{i_xq_xx}{g_x}\bigr\|_{L^{p,\infty}(g_x\cdot\mu_x)} \geq 1$.
It follows that  $\|Jx\| \geq \|x\|$ for all $x\in X$ in both cases.
\end{proof}

Proposition \ref{p3.3} has the following interpretation.  Assume that $1< p< \infty$.
Let $\cC_p$ be the class of all $p$-convex Banach lattices with constant $1$ and let $\cX_p$ be the class of all $L^p(\mu)$ spaces for any measure $\mu$. 
Then $\cC_p \subseteq \ol{\cX}_p$.  Since the reverse inclusion is evident, we have $\cC_p = \ol{\cX}_p$.
By Corollary \ref{c2.2}, for any Banach space $E$, $\rho_{\cC_p} = \rho_{\cX_p}$ on $\FVL[E]$. Note further that, since the measures $\mu\in \Gamma$  in Proposition \ref{p3.3} are finite, we could restrict the class $\cX_p$ to the class of $L^p(\mu)$ spaces with finite measures (let us denote this class by $\cX^F_p$) and the equality of norms would still hold. Actually, it can be proved by a standard argument that $\cX^F_p\subseteq \cX_p\subseteq \ol\cX^F_p$. The following result gives an alternative method for computing $\rho_{\cC_p}$.

\begin{thm}\label{t3.4}
Let $E$ be a Banach space and let $1< p<\infty$.
The norm $\rho_{\cC_p}$ on $\FVL[E]$ is given by
\[ \rho_{\cC_p}(f)  = \rho_p(f) : = \sup \|(f(x^*_1),\dots,f(x^*_n))\|_p,
\]
where the supremum is taken over all finite collections $x^*_1,\dots, x^*_n\in B_{E^*}$ satisfying the constraint
$\sup_{x\in B_E}\|(x^*_1(x),\dots, x^*_n(x))\|_p \leq 1$.
\end{thm}

\begin{proof}
As observed above, $\rho_{\cC_p} = \rho_{\cX_p}=\rho_{\cX^F_p}$, where $\cX^F_p$ is the class of all $L^p(\mu)$ spaces with finite $\mu$.
If $X\in \cX^F_p$ there exists a non-atomic finite measure $\mu$ so that $X$ is lattice isometric to a sublattice of $L^p(\mu)$. Thus, $\rho_{\{X\}} \leq \rho_{\{L^p(\mu)\}}$, where the two norms here are the ones associated to the classes with single elements, as indicated.
By Theorem \ref{t2.5},
$\rho_{\{L^p(\mu)\}}(f)  = \rho_p(f)$.
Finally, note that
\[ \rho_{\cX^F_p} =\sup_{X\in \cX^F_p}\rho_{\{X\}}\leq \sup_{\substack{\mu \text{ non-atomic}\\ \text{finite}}} \rho_{\{L^p(\mu)\}} = \rho_p \leq \rho_{\cX^F_p}.\]
Hence, $\rho_{\cC_p} = \rho_p$, as claimed.
\end{proof}

Let $\cC_{p,\infty}$ be the class of all $(p,\infty)$-convex Banach lattices with constant $1$ and let $\cX^\sigma_{p,\infty}$ be the class of all $L^{p,\infty}(\mu)$ spaces for any $\sigma$-finite measure $\mu$. Similarly to \Cref{t3.4}, we have the following result.

\begin{thm}\label{t3.5}
Let $E$ be a Banach space and let $1< p<\infty$.
The norm $\rho_{\cX^\sigma_{p,\infty}}$ on $\FVL[E]$ is given by
\[ \rho_{\cX^\sigma_{p,\infty}}(f)  = \rho_{p,\infty}(f) : = \sup \|(f(x^*_1),\dots,f(x^*_n))\|_{p,\infty},
\]
where the supremum is taken over all finite collections $x^*_1,\dots, x^*_n\in B_{E^*}$ satisfying the constraint
$\sup_{x\in B_E}\|(x^*_1(x),\dots, x^*_n(x))\|_{p,\infty} \leq 1$.
Moreover, $\rho_{p,\infty} \leq \rho_{\cC_{p,\infty}} \leq \gamma_p\,\rho_{p,\infty}$, where $\gamma_p =  (1-\frac{1}{p})^{\frac{1}{p}-1}$.
\end{thm}

\begin{proof}
The proof that $\rho_{\cX^\sigma_{p,\infty}} = \rho_{p,\infty}$ is similar to the proof of Theorem \ref{t3.4}.
Since $\cX^\sigma_{p,\infty} \subseteq \cC_{p,\infty}$, $ \rho_{\cX^\sigma_{p,\infty}} \leq \rho_{\cC_{p,\infty}}$.
On the other hand, let $X\in \cC_{p,\infty}$.  By Proposition \ref{p3.3}, there exists $Y\in \ol{\cX}^\sigma_{p,\infty}$ and a lattice isomorphism $J:X\to Y$ so that $\|x\| \leq \|Jx\| \leq \gamma_p\|x\|$ for all $x\in X$.
Let $T:E\to X$ be a linear contraction.  Then $\gamma_p^{-1}JT: E \to Y$ is a linear contraction.
Since $J$ is a lattice homomorphism, $\widehat{\gamma_p^{-1}JT} = \gamma_p^{-1}J\wh{T}$.
For any $f\in \FVL[E]$,
\[ \|\wh{T}f\| \leq \gamma_p\|\gamma_p^{-1}J\wh{T}f\| = \gamma_p\|\widehat{\gamma_p^{-1}JT}f\|\leq \gamma_p\rho_{\cX^\sigma_{p,\infty}}(f).
\]
Taking supremum over all linear contractions $T:E\to X$  for any $X\in \cC_{p,\infty}$ shows that $\rho_{\cC_{p,\infty}}\leq \gamma_p \rho_{\cX^\sigma_{p,\infty}}$.
\end{proof}

Observe that Theorems \ref{t3.4} and \ref{t3.5} say that $\rho_{\cC_p} =\rho_{\{\ell^p\}}$ and that
$\rho_{\{\ell^{p,\infty}\}} \leq \rho_{\cC_{p,\infty}} \leq \gamma_p\,\rho_{\{\ell^{p,\infty}\}}$.
As a result, $\fbl^{\cC_p}[E] = \fbl^{\{\ell^p\}}[E]$ and  $\fbl^{\cC_{p,\infty}}[E] = \fbl^{\{\ell^{p,\infty}\}}[E]$
and it is clear that both of these spaces are vector sublattices of $C(B_{E^*})$.
The proof works in exactly the same way if $p=1$, in which case we get that $\cC_1$, the class of all ($1$-convex) Banach lattices, is equal to $\ol{\cX}_1$, where $\cX_1$ consists of all $L^1(\mu)$ spaces.
The Banach lattice $\fbl^{\cC_1}[E]$ is the free Banach lattice generated by $E$, which is usually denoted by $\fbl[E]$.
From the above, we see that $\rho_{\cC_1} = \rho_{\cX_1} = \rho_{\{\ell^1\}}$.\medskip

\begin{remark}
    In this subsection we have made heavy use of factorization theorems through $L^p$ and weak-$L^p$ to identify the norm of the associated free spaces as a type of nonlinear weak-strong summing norm. We note that there is a large literature (see \cite{MR4081126, MR3712577, MR3011260,MR3248477} and related papers of these authors) on factorization theory and generalized summing operators  on  Banach function spaces, which may prove to be useful in the study of $\fbl^\mathcal{C}[E].$ There are also several  papers \cite{MR0599147, MR2591637, MR0604702,MR0613037,MR1333528,MR3900037} which study  factorizations of $p$-convex and $q$-concave operators. Although we will not do so here, it is certainly possible to extend some of these results to   other convexity/concavity conditions, such as $(p,\infty)$-convexity and $(q,1)$-concavity.
    \end{remark}

\subsection{Solution to Problem \ref{pr2.4} for the classes \texorpdfstring{$\cC_p$}{} and \texorpdfstring{$\cC_{p,\infty}$}{}.}\label{Solution to C=Cbar}
Let $\cC_p$ and $\cC_{p,\infty}$ be the class of $p$-convex Banach lattices and the class of Banach lattices satisfying an upper $p$-estimate with constant $1$, respectively.  
 We will characterize the classes of Banach lattices $\cD$ so that $(\fbl^{\cC_p}[E],\rho_{\cC_p})$ and $(\fbl^\cD[E], \rho_\cD)$ agree for any Banach space $E$ and also the classes $\cD'$ so that $(\fbl^{\cC_{p,\infty}}[E],\rho_{\cC_{p,\infty}})$ and $(\fbl^{\cD'}[E], \rho_{\cD'})$ are lattice isomorphic.
\medskip

For the proofs of the above claims, we need an extension of the isometric lattice-lifting property which may  be of independent interest. A Banach lattice $X$ is said to have the \emph{isometric lattice-lifting property} if there exists a lattice isometric embedding $\alpha$ from $X$ to $\fbl[X]$  such that $\wh{id_X}\alpha=id_X$. This property was introduced in \cite{AMRT}, inspired by previous works of Godefroy and Kalton \cite{GK} on Lipschitz-free spaces. Notably, Banach lattices ordered by a $1$-unconditional basis have this property.
\medskip

In \cite[Theorem 8.3]{OTTT} an alternative proof of the lattice lifting property for spaces  ordered by a $1$-unconditional basis was given, which also worked with $\fbl[E]$ replaced by $\fbp[E]$. We now show how to generalize these results to $\fbl^\cC[E]$. 

\begin{prop}\label{lattice lifting property}
    Let $X\in \ol{\cC}$ be a Banach space with a normalized 1-unconditional basis $(e_i)$, viewed as a Banach lattice in the pointwise order induced by the basis. Then $\fbl^\cC[X]$ contains a lattice isometric copy of $X$. Moreover, there exists a contractive lattice homomorphic projection onto this sublattice.
\end{prop}

\begin{proof}
  Using either \cite[Theorem 4.1]{AMRT} or \cite[Theorem 8.3]{OTTT} with $p=1$ we  obtain a lattice isometric embedding $\alpha:X\rightarrow \fbl^{\cC_1}[X]$. Since $\cC\subseteq \cC_1$, it follows that $\rho_\cC(f)\leq \rho_{\cC_1}(f)$ for every $f\in \FVL[X]$, so the formal identity extends to a norm one lattice homomorphism $j:\fbl^{\cC_1}[X]\rightarrow \fbl^\cC[X]$. On the other hand, since $X\in \ol\cC$, the identity operator on $X$ extends to a contractive lattice homomorphism $\wh{id_X}:\fbl^\cC[X]\rightarrow X$. It can be checked from the construction of $\alpha$ that $\wh{id_X}j\alpha=id_X$, from which it follows that
  \[\|x\|_X=\|\wh{id_X}j\alpha x\|\leq \rho_{\cC}(j\alpha x)\leq \rho_{\cC_1}(\alpha x)\leq \|x\|_X\]
  for every $x\in X$.
\end{proof}

\begin{thm}\label{t3.6}
The following are equivalent for a class of Banach lattices $\cD$.
\begin{enumerate}
\item For any Banach space $E$, $\fbl^{\cC_p}[E] = \fbl^\cD[E]$ as sets and the norms $\rho_{\cC_p}$ and $\rho_\cD$ agree there.
\item $\cD \subseteq \cC_p$ and $\ell^p \in \ol{\cD}$.
\end{enumerate}
\end{thm}

\begin{proof}
(1)$\implies$(2):  By  \Cref{lattice lifting property}, $\fbl^{\cD}[\ell^p] = \fbl^{\cC_p}[\ell^p]$ contains a lattice isometric copy of $\ell^p$. By Proposition \ref{p2.2}, $\fbl^\cD[\ell^p]\in \ol{\cD}$.
Hence, $\ell^p\in \ol{\cD}$.
\medskip

Suppose that $X\in \cD$.  Let $\vp_X:X\to \fbl^{\cD}[X]$ be the canonical embedding.
The identity $i: X\to X$ induces a lattice homomorphic contraction $\wh{i}:\fbl^\cD[X] \to X$ so that $\wh{i}\vp_X = i$.
Since $\fbl^\cD[X] = \fbl^{\cC_p}[X]$ lattice isometrically, this space belongs to $\ol{\cC}_p = \cC_p$ by Proposition \ref{p2.2}.
Hence, so does $\fbl^{\cD}[X]/\ker \wh{i}$.
Define $j: \fbl^\cD[X]/\ker \wh{i} \to X$  by $j[f] = \wh{i}f$, where $[f]$ is the equivalence class of $f$.
The map $j$ is a lattice isomorphism so that $\|j\| \leq 1$ and $x = j[\delta_x]$ for any $x\in X$.
Hence, $B_X = j(B_{\fbl^\cD[X]/\ker \wh{i}})$.  This shows that $j$ is an onto lattice isometry.
Therefore, $X\in \cC_p$.

\medskip

\noindent (2)$\implies$(1):
Since $\cD\subseteq \cC_p$, for any Banach space $E$, $\rho_\cD \leq \rho_{\cC_p}$ on $\FVL[E]$.
Similarly, since $\ell^p\in \ol{\cD}$, $\rho_{\{\ell^p\}}\leq \rho_{\ol{\cD}}$.
However, $\rho_{\cC_p} = \rho_{\{\ell^p\}}$ and $ \rho_{\ol{\cD}} = \rho_\cD$ by Corollary \ref{c2.2}.
Therefore, we have the reverse inequality $\rho_{\cC_p}\leq \rho_\cD$.
\end{proof}

For the case of $\cC_{p,\infty}$, the following result can be proven in  essentially the same way as  \Cref{t3.6}. The only difference is that, since $\ell^{p,\infty}$ does not have a basis,  \Cref{lattice lifting property} cannot be applied directly. However, we can still use \Cref{lattice lifting property} to embed $\ell^{p,\infty}$ lattice isometrically into $(\bigoplus_m\fbl^{\cC_{p,\infty}}[\ell^{p,\infty}(m)])_\infty$ and continue with the rest of the proof with  slight adaptations.
\begin{thm}\label{t3.7}
The following are equivalent for a class of Banach lattices $\cD$.
\begin{enumerate}
\item There is a constant $M_1 <\infty$ so that for any Banach space $E$, $\fbl^{\cC_{p,\infty}}[E] = \fbl^\cD[E]$ as sets and  $\frac{1}{M_1}\rho_\cD \leq \rho_{\cC_{p,\infty}}\leq M_1\rho_\cD$ there.
\item There is a constant $M_2 <\infty$ so that any $X\in \cD$ satisfies an upper $p$-estimate with constant $M_2$, and
 $\ell^{p,\infty}$ is $M_2$-lattice isomorphic to a space $Y \in \ol{\cD}$.
\end{enumerate}
\end{thm}

\section{Distinguishing \texorpdfstring{$\rho_{\cX_{p,\infty}}$}{} and \texorpdfstring{$\rho_{\cC_{p,\infty}}$}{}}\label{Dist norm}

A natural question that arises from Proposition \ref{p3.3} and Theorem \ref{t3.5}
is  whether the constant $\gamma_p$ can be improved to $1$.
In this section, we will show that the answer is negative in general.  In particular, there is a Banach space $E$ so that $\rho_{\cC_{p,\infty}}\neq \rho_{\cX_{p,\infty}}$ on $\FVL[E]$.
\medskip

Let $X, Y$ be Banach lattices and let $1\leq C<\infty$.  We say that $X$ $C$-lattice embeds into $Y$ if there exists a lattice isomorphic injection $T:X\to Y$ such that $C^{-1}\|x\| \leq \|Tx\| \leq \|x\|$ for all $x\in X$.
In this case, we call $T$ a $C$-lattice embedding.  If $C =1$, then $T$ is a lattice isometric embedding.
A Banach space with a normalized $1$-unconditional basis is regarded as a Banach lattice in the coordinate-wise order.

\begin{thm}\label{t4.1}
Let $X$ be a Banach lattice with a normalized $1$-unconditional basis $(e_i)$.
Denote the biorthogonal functionals by $(e_i^*)$.  For any $C\geq 1$, the following statements are equivalent.
\begin{enumerate}
\item For each normalized finitely supported $a = \sum^n_{i=1}a_ie_i\in X_+$ and every $\ep >0$, there exist $b = \sum^n_{i=1}b_ie^*_i\in X^*_+$ and $d = (d_i)^n_{i=1}\in \R^n_+$ so that $\sum^n_{i=1}a_ib_i > 1-\ep$,
$\sum^n_{i=1}d_i=1$
and  $\|\sum_{i\in I}b_ie^*_i\|^{p'} \leq C^{p'}\sum_{i\in I}d_i$ for any $I \subseteq \{1,\dots,n\}$.
\item There is a set $\Gamma$ of (probability) measures such that $X$ $C$-lattice  embeds into $(\bigoplus_{\mu\in \Gamma}L^{p,\infty}(\mu))_\infty$.
\end{enumerate}
\end{thm}

\begin{proof}
(1)$\implies$(2):
Let $a\in X_+$ be normalized and finitely supported, and let $\ep >0$ be given.  Obtain  $b$ and $d$ from (1). Let $(U_i)^n_{i=1}$ be a measurable partition of $(0,1)$ such that $\lambda(U_i) = d_i$, $1\leq i\leq n$, where $\lambda$ denotes the Lebesgue measure.
Define
\[ S_{a,\ep}: X\to L^{p,\infty}(0,1) \text{  by } S_{a,\ep}(\sum^\infty_{i=1}c_ie_i) = \sum^n_{i=1}\frac{c_ib_i\chi_{U_i}}{d_i}.\]
We have
\[ \|S_{a,\ep} a\| \geq \lambda(0,1)^{-\frac{1}{p'}}\int^1_0|S_{a,\ep} a|\,d\lambda = \sum^n_{i=1}a_ib_i > 1-\ep.\]
On the other hand, if $\|\sum^\infty_{i=1}c_ie_i\|\leq1$ then there exists $I \subseteq \{1,\dots, n\}$ so that
\begin{align*}
 \|S_{a,\ep} c\| &= \lambda (\bigcup_{i\in I}U_i)^{-\frac{1}{p'}} \int_{\bigcup_{i\in I}U_i}|S_{a,\ep}c|\,d\lambda=(\sum_{i\in I}d_i)^{-\frac{1}{p'}}\sum_{i\in I}|c_i|b_i\\& = (\sum_{i\in I}d_i)^{-\frac{1}{p'}}\la \sum^\infty_{i=1}|c_i|e_i, \sum_{i\in I}b_ie^*_i\ra
\leq (\sum_{i\in I}d_i)^{-\frac{1}{p'}}\|\sum_{i\in I}b_ie^*_i\| \leq C.
\end{align*}
Let $\Gamma$ be the Cartesian product of the set of normalized finitely supported elements in $X_+$ and the interval $(0,\frac{1}{2})$.
It is now clear that the map $T:X\to (\bigoplus_{(a,\ep)\in \Gamma}L^{p,\infty}(0,1))_\infty$
given by
\[ T(\sum^\infty_{i=1}c_ie_i) = (C^{-1}S_{a,\ep}c)_{(a,\ep)\in \Gamma}\]
is a $C$-lattice embedding.

\bigskip

\noindent (2)$\implies$(1):  Let $T:X\to  (\bigoplus_{\mu\in \Gamma}L^{p,\infty}(\mu))_\infty$ be a $C$-lattice  embedding.
Suppose that $a = \sum^n_{i=1}a_ie_i$ is normalized  and $\ep > 0$.
For each $\mu_0\in \Gamma$, let $P_{\mu_0}$ be the projection of  $(\bigoplus_{\mu\in \Gamma}L^{p,\infty}(\mu))_\infty$ onto the $\mu_0$-th component.
There exists $\mu_0$ such that $C\|P_{\mu_0}Ta\| > 1-\ep$.
Let  $f_i = P_{\mu_0}Te_i.$
Choose a $\mu_0$-measurable set $U$ such that $0 < \mu_0(U) < \infty$ and
\[ C\int_U P_{\mu_0}Ta\,d\mu_0 > (1-\ep)(\mu_0(U))^{\frac{1}{p'}}.\]
Set $U_i = U \cap \supp f_i$ and $m = (\sum^n_{i=1}\mu_0(U_i))^{\frac{1}{p'}}$. Define
\[ b_i = \frac{C\int_{U_i}f_i\,d\mu_0}{m}, \quad d_i = \frac{\mu_0(U_i)}{m^{p'}}, \quad 1\leq i\leq n.\]
We have $d_i \geq 0$, $\sum^n_{i=1}d_i =1$ and
\[
 \sum^n_{i=1}a_ib_i  = \frac{C}{m}\int_{\bigcup^n_{i=1}U_i}\sum^n_{i=1}a_if_i\,d\mu_0 =
 \frac{C}{m}\int_{\bigcup^n_{i=1}U_i}P_{\mu_0}Ta > 1-\ep.
\]
On the other hand, suppose that $I\subseteq \{1,\dots, n\}$. If $\|\sum^\infty_{i=1}c_ie_i\| \leq 1$, then
\begin{align*}
\la \sum^\infty_{i=1}c_ie_i, \sum_{i\in I}b_ie^*_i\ra & = \sum_{i\in I}c_ib_i = \frac{C}{m}\int_{\bigcup_{i\in I}U_i}\sum_{i\in I}c_if_i\,d\mu_0\\
& \leq  \frac{C}{m}\|P_{\mu_0}T(\sum^\infty_{i=1}c_ie_i)\|\,(\mu_0(\bigcup_{i\in I}U_i))^{\frac{1}{p'}}
\leq C(\sum_{i\in I}d_i)^{\frac{1}{p'}}.
\end{align*}
This proves that $\|\sum_{i\in I}b_ie_i^*\|^{p'} \leq C^{p'}\sum_{i\in I}d_i$.
\end{proof}

\begin{remark}
If $X$ is a two dimensional Banach lattice that satisfies an upper $p$-estimate with constant one, then $C=1$ works in \Cref{t4.1}.
Indeed, let $(e_1,e_2)$ be a normalized $1$-unconditional basis for $X$.  Suppose that $a= a_1e_1 + a_2e_2$ is normalized and non-negative.   Choose $b = b_1e_1^* + b_2e^*_2\in X^*_+$ such that $\la a,b\ra =1 = \|b\|$. Define $d_i = b_i^{p'}/(b_1^{p'} + b_2^{p'})$, $i=1,2$.
Note that $X^*$ satisfies a lower $p'$-estimate with constant one and $e_1^*$ and $e_2^*$ are disjoint, so
\[ b_1^{p'} + b_2^{p'} = \|b_1e^*_1\|^{p'} + \|b_2e^*_2\|^{p'} \leq \|b\|^{p'} =1.\]
It is easy to check that $\|\sum_{i\in I}b_ie^*_i\|^{p'} \leq \sum_{i\in I}d_i$ for any $I \subseteq \{1,2\}$ and $d_1 + d_2 = 1$.
\end{remark}

\begin{cor}\label{c4.2}
Suppose that $X$ is a  Banach lattice with a normalized  $1$-unconditional basis $(e_i)$ that $C$-lattice embeds into $(\bigoplus_{\mu\in \Gamma}L^{p,\infty}(\mu))_\infty$.
For any normalized $b = \sum^n_{i=1}b_ie^*_i \in X^*_+$ and any  $I_j\subseteq \{1,\dots, n\}$, $1\leq j\leq m$, such  that $\sum^m_{j=1}\chi_{I_j} = l\chi_{\{1,\dots,n\}}$ for some $l\in \N$, we have
\[ \sum_{j=1}^m\|\sum_{i\in I_j}b_ie^*_i\|^{p'} \leq C^{p'}l.\]
\end{cor}

\begin{proof}
 Choose a normalized $a = \sum^n_{i=1}a_ie_i\in X_+$ so that $\sum^n_{i=1}a_ib_i =1$.  First assume that $[(e_i)^n_{i=1}]$ is uniformly convex.
Let $\ep > 0$ be given. There exists $\delta >0$ so that if $b' = \sum^n_{i=1}b'_ie^*_i\in X^*_+$, $\|b'\|=1$
and $\la a, b'\ra > 1-\delta$, then $\|b-b'\| < \ep$.
By \Cref{t4.1}, there exists $b'\in X^*_+$ and $d = (d_i)^n_{i=1}\in \R^n_+$ so that $\la a,b'\ra > 1-\delta$ and  $\|\sum_{i\in I}b_i'e_i^*\|^{p'} \leq C^{p'}\sum_{i\in I}d_i$ for any $I\subseteq \{1,\dots,n\}$.
It follows that for any $I\subseteq \{1,\dots, n\}$,
\[ \|\sum_{i\in I}b_ie_i^*\|  \leq  \|\sum_{i\in I}|b_i-b_i'|e_i^*\|+ \|\sum_{i\in I}b'_ie_i^*\| \leq \ep\cdot \#I + C(\sum_{i\in I}d_i)^{\frac{1}{p'}}.
\]
Now assume that $I_j\subseteq \{1,\dots, n\}$, $1\leq j\leq m$, and that $\sum^m_{j=1}\chi_{I_j} = l\chi_{\{1,\dots,n\}}$ for some $l\in \N$.
Set $A_j = \ep\cdot \#I_j$ and $B_j = C(\sum_{i\in I_j}d_i)^{\frac{1}{p'}}$.  Then
\begin{align*}(\sum_{j=1}^m&\|\sum_{i\in I_j}b_ie^*_i\|^{p'})^{\frac{1}{p'}}\leq (\sum^m_{j=1}(A_j+ B_j)^{p'})^{\frac{1}{p'}}
\leq (\sum^m_{j=1}A_j^{p'})^{\frac{1}{p'}} + (\sum^m_{j=1}B_j^{p'})^{\frac{1}{p'}} \\
& \leq \sum^m_{j=1}A_j + (\sum^m_{j=1}B_j^{p'})^{\frac{1}{p'}}
= \ep\sum^m_{j=1}\# I_j  + C(\sum^m_{j=1}\sum_{i\in I_j}d_i)^{\frac{1}{p'}}
\leq \ep nl + Cl^{\frac{1}{p'}}.
\end{align*}
Taking $\ep \downarrow 0$ gives the desired result.
\medskip

We now return to a general $X$ with a normalized $1$-unconditional basis $(e_i)$.  Given a normalized $b = \sum^n_{i=1}b_ie^*_i \in X^*_+$ and  $\ep > 0$, there is a uniformly convex lattice norm $|||\cdot|||$ on $[(e_i)^n_{i=1}]$ such that
$\|x\| \leq |||x||| \leq (1+\ep)\|x\|$ for any $x\in [(e_i)^n_{i=1}]$.
Set $u_i = e_i/|||e_i|||$, $u_i^* = e_i^*/|||e_i^*|||$ and $b' = \sum^n_{i=1}b_i'u^*_i = b/|||b|||$.
Note that $(u_i)^n_{i=1}$ is a $1$-unconditional basis for $[(e_i)^n_{i=1}] =   [(u_i)^n_{i=1}]$ and $b'$ is $|||\cdot|||$-normalized.   Moreover, if $T: [(e_i)^n_{i=1}]\to (\bigoplus_{\mu\in \Gamma}L^{p,\infty}(\mu))_\infty$ is a $C$-lattice embedding, then
\[ \ti{T}x = Tx: [(u_i)^n_{i=1}]\to (\bigoplus_{\mu\in \Gamma}L^{p,\infty}(\mu))_\infty\]
is a $(1+\ep)C$-lattice embedding.
Assume that $I_j\subseteq \{1,\dots, n\}$, $1\leq j\leq m$, such  that $\sum^m_{j=1}\chi_{I_j} = l\chi_{\{1,\dots,n\}}$ for some $l\in \N$.
By the first part,
\[  \sum_{j=1}^m|||\sum_{i\in I_j}b_i'u^*_i|||^{p'} \leq (1+\ep)^{p'}C^{p'}l.
\]
Since $|||b|||\leq 1+\ep$, for any $I\subseteq \{1,\dots,n\}$ we have
\[ |||\sum_{i\in I}b'_iu^*_i||| =\frac{1}{|||b|||} |||\sum_{i\in I}b_ie^*_i||| \geq \frac{1}{1+\ep}\|\sum_{i\in I}b_ie^*_i\|.
\]
Thus,
\[ \sum^m_{j=1}\|\sum_{i\in I}b_ie_i^*\|^{p'} \leq (1+\ep)^{2p'}C^{p'}l.\]
Taking $\ep\downarrow 0$ gives the desired result.
\end{proof}

The next example shows that $\gamma_p$ cannot be made to be $1$ in Proposition \ref{p3.3}.

\begin{ex}\label{e4.3} Suppose that $1< p < \infty$. Denote by $C_p$ the infimum of the set of constants $C$   so that every $(p,\infty)$-convex Banach lattice with constant one $C$-lattice embeds into $(\bigoplus_{\gamma\in \Gamma}L^{p,\infty}(\mu))_\infty$ for some set of measures $\Gamma$.
Then $C_p^{p'} \geq  \frac{3 \cdot 2^{p'}}{2(1+2^{p'})}> 1$.
\end{ex}

\begin{proof}
Let $X = \R^3$ so that the norm on $X^*$  is
\[ \|(b_1,b_2,b_3)\|_{X^*} = \max\bigl\{\bigl(|b_i|^{p'} + (|b_j|+|b_k|)^{p'}\bigr)^{\frac{1}{p'}}: \{i,j,k\} = \{1,2,3\}\bigr\}.\]
It is clear that the coordinate unit vectors form a normalized $1$-unconditional basis for $X$.
Suppose that $b=(b_1,b_2,b_3)\in X^*$ and $I,J$ are disjoint non-empty  subsets of $\{1,2,3\}$. Without loss of generality, we may assume that $I= \{i,j\}$ and $J = \{k\}$, where $i,j,k$ are distinct. Then
\[ \|b\chi_I\|^{p'}_{X^*} + \|b\chi_J\|^{p'}_{X^*} = (|b_i| + |b_j|)^{p'} + |b_k|^{p'} \leq \|b\|_{X^*}^{p'}.\]
Thus, $X^*$ satisfies a lower $p'$-estimate with constant $1$.  Hence, $X$ satisfies an upper $p$-estimate with constant $1$.
Suppose that $X$ $C$-lattice embeds into some $(\bigoplus_{\gamma\in \Gamma}L^{p,\infty}(\mu))_\infty$.
The vector $(1+2^{p'})^{-\frac{1}{p'}}(1,1,1)$ is normalized in $X^*_+$.
By \Cref{c4.2},
\[ \frac{1}{1+2^{p'} }\,(\|(1,1,0)\|^{p'} + \|(1,0,1)\|^{p'} + \|(0,1,1)\|^{p'}) \leq 2C^{p'}.\]
Thus,
\[ \frac{3 \cdot 2^{p'}}{2(1+2^{p'})} \leq C^{p'}.
\]
Taking infimum over the  eligible $C$'s gives the desired result. Direct verification shows that  $\frac{3 \cdot 2^{p'}}{2(1+2^{p'})}> 1$ if $1< p < \infty$.
\end{proof}
Now we are ready to show that $\gamma_p$ also cannot be made to be $1$ in Theorem \ref{t3.5}. Note that the following result can be adapted to work for other classes of lattices $\mathcal{C}$.
\begin{prop}\label{thmembed}
    Let $X$ be a Banach lattice with a normalized 1-unconditional basis $(e_i)$ satisfying upper $p$-estimates. Let $C\geq1$. The following are equivalent.
    \begin{enumerate}
        \item $X$ $C$-lattice embeds in $(\bigoplus_{\gamma\in \Gamma}L^{p,\infty}(\mu))_\infty$.
        \item For every operator $T:X\rightarrow X$, the extension $\widehat T:\fbl^{\cX_{p,\infty}}[X]\rightarrow X$ satisfies $\|\widehat T\|\leq C\|T\|.$
        \item The extension $\widehat{id_X}:\fbl^{\cX_{p,\infty}}[X]\rightarrow X$ of the identity $id_X:X\rightarrow X$ satisfies $\|\widehat{id_X}\|\leq C.$
    \end{enumerate}
\end{prop}

\begin{proof}
  Clearly, $(1)\Rightarrow(2)\Rightarrow (3)$.
\medskip

    For the proof of $(3)\Rightarrow (1)$ we first note that by the remark after Theorem \ref{t3.5}, for any Banach space $E$, $\rho_{\cX^\sigma_{p,\infty}} = \rho_{\{\ell^{p,\infty}\}}$ on $\FVL[E]$.
Thus, it follows from  Proposition \ref{p2.2} that $\fbl^{\cX_{p,\infty}}[E]$ belongs to the class $\ol{\{\ell^{p,\infty}\}}$, i.e., it lattice isometrically embeds into $(\bigoplus_{\gamma \in \Gamma}\ell^{p,\infty})_\infty$.
\medskip

 We now adapt the proof of  \Cref{lattice lifting property} to our setting. Recall that by \cite[Theorem 8.3]{OTTT} there is a lattice isometric embedding $\alpha:X\rightarrow \fbl[X]=\fbl^{\cC_1}[X]$ such that $\widehat{id_X}\alpha=id_X$. Moreover, since $\cX^\sigma_{p, \infty}\subseteq \cC_1$, the formal identity $j:\fbl^{\cC_1}[X]\rightarrow \fbl^{\cX^\sigma_{p,\infty}}[X]$ is a continuous norm one injection such that $\widehat{id_X}j\alpha=id_X$. We claim that $j\alpha:X\rightarrow \fbl^{\cX^\sigma_{p,\infty}}[X]$ defines a $C$-lattice embedding. If true, the previous paragraph gives the implication $(3)\Rightarrow (1)$. To prove the claim, simply note that for every $x\in X$
 \[\frac{1}{C}\|x\|\leq  \frac{1}{\|\widehat{id_X}\|}\|\widehat{id_X}j\alpha x\| \leq \rho_{\cX_{p\infty}}(j\alpha x)\leq \rho_{\cC_1}(\alpha x) \leq \|x\|.\]
\end{proof}

\subsection{Equivalent renormings of \texorpdfstring{$L^{p,\infty}(\mu)$}{}}\label{appendix weak Lp}\label{lattice renorm}
Let $1<p< \infty$. As mentioned above, given a measure space $(\Omega,\Sigma,\mu)$, the space $L^{p,\infty}(\mu)$ is the set of all measurable functions $f:\Omega\rightarrow \R$ such that
\[ \|f\|^*_{L^{p,\infty}}:=\sup_{t>0} t\mu(\{\,|f|>t\})^{\frac{1}{p}} \]
is finite. In the previous subsections, we equipped $L^{p,\infty}(\mu)$ with the norm
\[ \|f\|_{L^{p,\infty}} = \sup \{\mu(E)^{\frac{1}{p}-1}\int_E |f| d\mu: 0 < \mu(E) <\infty\}.\]
However, for each $1\leq r<p$, we could have equally chosen to equip $L^{p,\infty}(\mu)$ with the norm 
\[    \|f\|_{L^{p,\infty}_r}:=\sup_{0<\mu(A)<\infty} \mu(A)^{-\frac{1}{r}+\frac{1}{p}} \left(\int_A |f|^r\,d\mu\right)^{\frac{1}{r}}. \]
Indeed, it is  well-known (cf. \cite[Exercise 1.1.12]{GrafakosCFA}) that
\[\|f\|^*_{L^{p,\infty}} \leq \|f\|_{L^{p,\infty}_r} \leq \left(\frac{p}{p-r}\right)^{\frac{1}{r}}  \| f\|^*_{L^{p,\infty}}\]
for every $f\in L^{p,\infty}(\mu)$, and it is easy to see that for each $1\leq r<p$,  $\|\cdot\|_{L^{p,\infty}_r}$ defines a lattice norm on $L^{p,\infty}(\mu)$ with upper $p$-estimate constant $1$.  For this reason, it is natural to ask whether the ambiguity in choosing a norm on $L^{p,\infty}(\mu)$ is the reason why $\gamma_p$ cannot be chosen to be $1$ in Proposition \ref{p3.3} and Theorem \ref{t3.5}.
\medskip

The aim of this subsection is to show that $(L^{p,\infty}(\mu),\|\cdot\|_{L^{p,\infty}_r})$ lattice isometrically  embeds into a Banach lattice of the form
$(\bigoplus_{\mu\in \Gamma}L^{p,\infty}(\nu))_\infty$, where each $L^{p,\infty}(\nu)$ is given the $L^{p,\infty}_1$-norm. As a consequence, we deduce that $\rho_{p,\infty}(\cdot)=\rho_{L^{p,\infty}_1}(\cdot)\geq \rho_{L^{p,\infty}_r}(\cdot)$ on $\FVL[E]$, proving that the above renormings cannot eliminate the constant $\gamma_p$ in  Proposition \ref{p3.3} and Theorem \ref{t3.5}.
\medskip

We begin with a lemma.

\begin{lem}\label{lem A2}
Let $e = (1,\dots, 1)\in \R^n$ and let $(e_i)^n_{i=1}$ be the coordinate unit vector basis for $\R^n$. Let $\beta, b\in \R^n_+$, $\|\beta\|_1, \|b\|_1 \leq 1$, and let $0 < s < 1$. Define $\al : \R^n_+ \to \R$ by
$\al(x) = (\beta\cdot x)^{1-s}\, (b\cdot x)^s$, where $\cdot$ is the dot product on $\R^n$. 
Then $\al$ is positively homogeneous, $\al(e) = \|\beta\|_1^{1-s}\cdot \|b\|_1^s \leq 1$ and $\sum^m_{j=1}\al(x_j) \leq \al(\sum^m_{j=1}x_j)$ for any sequence $(x_j)^m_{j=1} \subseteq \ell^1(n)_+$.
\end{lem}

\begin{proof}
Let $(x_j)^m_{j=1} \subseteq \R^n_+$.
Then
\begin{align*}
\sum^m_{j=1}\al(x_j) & \leq \|(\beta\cdot x_j)^{1-s}\|_{\frac{1}{1-s}}\, \|(b\cdot x_j)^s\|_{\frac{1}{s}}\\
& = (\beta\cdot \sum^m_{j=1}x_j)^{1-s}\, (b\cdot \sum^m_{j=1}x_j)^{s} = \al(\sum^m_{j=1}x_j).
\end{align*}
The remaining assertions are clear.
\end{proof}

Let $x\in \R^n$ be expressed as $(x(j))^n_{j=1}$. In our next result, for fixed $i\in \N$, we write $x < i$ if $\max\supp x <i$.

\begin{prop}\label{prop A3}
There is a sequence $(d_i)^n_{i=1}$ of non-negative real numbers such that for each $i$,
\[ \sum^{i-1}_{j=1}x(j)d_j +\al(z) - m -\sum^{i-1}_{j=1}y(j)d_j \leq d_i \leq m' + \sum^{i-1}_{j=1}y'(j)d_j
- \sum^{i-1}_{j=1}x'(j)d_j - \al(z')\]
provided that $m, m' \in \R_+$, $x,y,z, x',y',z' \in \R^n_+$,
\begin{equation}\label{e4.1} x,y,x',y' < i,\quad  x+ z = e_i + me + y \quad \text{and} \quad e_i + x' +z' =  m'e+ y'.\end{equation}
\end{prop}

\begin{proof}
The proof is by induction on $i$.
For $i=1$, $x = y = x' = y'= 0$.  Thus it suffices to show that
 \[\al(z) - m \leq m' -\al(z') \quad\text{if}\quad z= e_1 + me\ \ \text{and}\  \ e_1 + z' = m'e.\]
To see this, note that
\[ m+m' \geq  \al((m+m')e) = \al(z+ z') \geq \al(z) + \al(z'),
\]
from which the desired conclusion follows.
\medskip

Assume now that $d_j$, $1 \leq j < i$, have been chosen to satisfy the proposition. To be able to choose $d_i$ satisfying  the proposition, we need to show that
\[ \sum^{i-1}_{j=1}x(j)d_j +\al(z) - m -\sum^{i-1}_{j=1}y(j)d_j \leq m' + \sum^{i-1}_{j=1}y'(j)d_j
- \sum^{i-1}_{j=1}x'(j)d_j - \al(z')\]
if  $m,m',x,y,z,x',y',z'$   satisfy (\ref{e4.1}).
Equivalently, we must show that
\begin{equation}\label{e4.2}\sum^{i-1}_{j=1}(x(j)+x'(j))d_j  + \al(z) + \al(z') \leq m + m' + \sum^{i-1}_{j=1}(y(j)+y'(j))d_j.\end{equation}
If $x+x' = y+ y'$, then (\ref{e4.1}) gives $z+ z' = (m+m')e$.
Thus, $m+m' \geq \al(z) + \al(z')$, yielding (\ref{e4.2}). Otherwise, let
  $u = (x+ x') \wedge (y+y')\in \R^n_+$.  Then   $u <i$ and
 $v_1:= x+ x' - u,\ v_2:=y+y'-u$ are  disjoint in $\R^n_+$, not both $0$.
Assume that  $j_0 = \max \supp v_1 > \max\supp v_2$, as the other case is similar.
Set $w = v_1 - v_1(j_0)e_{j_0}$.
Then $w, v_2 < j_0$, $w(j)=v_1(j)$ if $j<j_0$ and
\[ v_1(j_0)e_{j_0} + w + (z+z') =   (m+m')e + v_2.
\]
Using the right  inequality in the inductive hypothesis and the positive homogeneity of $\al$, we see that
\[  v_1(j_0) d_{j_0} \leq (m+m') +\sum^{j_0-1}_{j=1}v_2(j)d_j  -\sum^{j_0-1}_{j=1}v_1(j)d_j-  \al(z+z').  \]
Adding $\sum^{i-1}_{j=1}u(j)d_j$ to both sides and rearranging  gives
\[ \sum^{i-1}_{j=1}(x(j)+x'(j))d_j  + \al(z+z') \leq m + m' + \sum^{i-1}_{j=1}(y(j)+y'(j))d_j.\]
The desired inequality (\ref{e4.2}) then follows since $\al(z) + \al(z') \leq \al(z+ z')$.
\end{proof}

\begin{cor}
In the above notation, we have  $d= (d_i)^n_{i=1} \in \R^n_+$, $\al(e) \leq \|d\|_1 \leq 1$ and $\al(x) \leq x\cdot d$
for any $x\in \R^n_+$.
\end{cor}

\begin{proof}
Taking $i = n$ in \Cref{prop A3}, we have
\[x = 0,\ y = e -e_n,\ z = e,\ m=0;\ x' = e-e_n,\ y' = 0,\ z' = 0,\ m'  = 1.\]
Note that $x, y, x', y' < n$, $x+z = e_n +me+y$ and $e_n + x' + z'= m'e+y'$.
Thus,
\[ \al(e) -\sum^{n-1}_{j=1}d_j \leq d_n \leq 1
- \sum^{n-1}_{j=1}d_j.\]
Hence, $\al(e) \leq \|d\|_1 \leq 1$.   Finally, if $u\in \R^n_+$, let $x = 0,\ z = u,\ m=0$ and $y = u-u(n)e_n$.
Then
 $x+ z = u(n)e_n + me + y$.  By Proposition \ref{prop A3} and the positive homogeneity of $\al$,
\[ \al(u) -\sum^{n-1}_{j=1}u(j)d_j \leq u(n)d_n \quad \implies \quad \al(u) \leq u\cdot d.\]
\end{proof}


\begin{lem}
Let $(\Om,\Sigma,\mu)$ be a measure space. Suppose that $(a_i)^n_{i=1} \in \R^n_+$ and $(U_i)^n_{i=1}$ is a disjoint sequence in $\Sigma$ with $0< \mu(U_i)<\infty$ so that $C: =(\sum^n_{i=1}\mu(U_i))^{-\frac{1}{r}+\frac{1}{p}} \|\sum^n_{i=1}a_i\chi_{U_i}\|_{L^r(\mu)} \leq 1$.
There is a measure $\nu$ on $\Sigma$ and a linear lattice homomorphism
$S: L^{p,\infty}_r(\mu) \to L^{p,\infty}_1(\nu)$ so that $\|S\| \leq 1$ and $\|S(\sum^n_{i=1}a_i\chi_{U_i})\| \geq C^r$.
\end{lem}

\begin{proof}
Define $M = \sum^n_{j=1}\mu(U_j)$, $b_i = \frac{\mu(U_i)}{M}$, $b= (b_i)^n_{i=1}$ and
$\beta_i = M^{\frac{r}{p}-1}\mu(U_i)a_i^r$, $\beta = (\beta_i)^n_{i=1}$.  Clearly, $\beta, b\in \R^n_+$ and $\|b\|_1 =1$.
Moreover,
\[ \|\beta\|_1 =M^{r(\frac{1}{p}-\frac{1}{r})}\,\|\sum^n_{i=1}a_i\chi_{U_i}\|^r_{L^r(\mu)} = C^r \leq 1.
\]
Take $s = p'(\frac{1}{r}-\frac{1}{p})$ and define $\al$ as in Lemma \ref{lem A2}. Note that $\al(e) = \|\beta\|_1^{1-s}\,\|b\|_1^s = C^{r(1-s)}\leq 1$. Obtain $d$ using Proposition \ref{prop A3}.
Define a measure $\nu$ on $\Sigma$ by $\nu(A) = \sum^n_{i=1}\frac{d_i\,\mu(A\cap U_i)}{\mu(U_i)}$ and a linear operator
\[ S: L^{p,\infty}_r(\mu) \to L^{p,\infty}_1(\nu) \quad \text{by}\quad Sf = M^{\frac{r}{p}}\sum^n_{i=1}\frac{a_i^{r-1}b_i}{d_i}f\chi_{U_i}.\]
Clearly, $S$ is a linear lattice homomorphism.
Note that $\nu(\bigcup^n_{i=1}U_i) = \sum^n_{i=1}d_i \leq 1$.  Hence,
\begin{align*} \|S(\sum^n_{i=1}a_i\chi_{U_i})\| &\geq   M^{\frac{r}{p}}\|\sum^n_{i=1}\frac{a_i^rb_i}{d_i}\chi_{U_i}\|_{L^1(\nu)} = M^{\frac{r}{p}}\sum^n_{i=1}a^r_ib_i \\&= M^{\frac{r}{p}-1}\|\sum^n_{i=1}a_i\chi_{U_i}\|^r_{L^r(\mu)} = C^r.
\end{align*}
On the other hand, let $f\in L^{p,\infty}_r(\mu)$ with $\|f\| \leq 1$.  If $A\in \Sigma$ with $0 < \nu(A) <\infty$ then
\begin{align*}
 \int_A|Sf| \,d\nu & = M^{\frac{r}{p}}\sum^n_{i=1}\frac{a_i^{r-1}b_i}{d_i}\int_{A\cap U_i} |f| \,d\nu
=  M^{\frac{r}{p}-1}\sum^n_{i=1} a_i^{r-1}\int_{A\cap U_i} |f| \,d\mu\\
&\leq M^{\frac{r}{p}-1}\bigl\|\sum^n_{i=1}a^{r-1}_i\chi_{A\cap U_i}\bigr\|_{L^{r'}(\mu)}\, \bigl\|f\chi_{\bigcup^n_{i=1}(A\cap U_i)}\bigr\|_{L^r(\mu)}\\
&\leq  \left(\sum^n_{i=1}\beta_i\frac{\,\mu(A\cap{U_i})}{\mu(U_i)}\right)^{\frac{1}{r'}}\, \left(\sum^n_{i=1} b_i\frac{\,\mu(A\cap{U_i})}{\mu(U_i)}\right)^{\frac{1}{r}-\frac{1}{p}}.
\end{align*}
Let $x = \bigl(\frac{\mu(A\cap U_i)}{\mu(U_i)}\bigr)^n_{i=1} \in \R^n_+$.
By direct calculation, $1-s =\frac{p'}{r'}$.  Thus,
\[ \bigl|\int_ASf \,d\nu\bigr| \leq (\beta\cdot x)^{\frac{1-s}{p'}}\, (b\cdot x)^{\frac{s}{p'}} = \al(x)^{\frac{1}{p'}} \leq (x\cdot d)^{\frac{1}{p'}} = \nu(A)^{\frac{1}{p'}}.
\]
Hence, $\|Sf\| \leq 1$.
\end{proof}

\begin{thm}
Suppose that $(\Om,\Sigma,\mu)$ is a measure space and  $1< r < p$.
There is a set $\Gamma$ of measures on $\Sigma$ so that $L^{p,\infty}_r(\mu)$ lattice isometrically embeds into $(\bigoplus_{\nu \in \Gamma} L^{p,\infty}_1(\nu))_{\infty}.$
\end{thm}

\begin{proof}
Let $\Gamma$ be the set of simple functions $f = \sum^n_{i=1}a_i\chi_{U_i}$, where $n\in \N$, $a_i >0$, $0< \mu(U_i) <\infty$ and $(U_i)^n_{i=1}$ is a disjoint sequence in $\Sigma$ with
\[ C_f = (\sum^n_{i=1}\mu(U_i))^{-\frac{1}{r}+\frac{1}{p}}\|\sum^n_{i=1}a_i\chi_{U_i}\|_r \leq 1.\]
For each $f\in \Gamma$, use the above results to choose a lattice homomorphism $S: L^{p,\infty}_r(\mu) \to L^{p,\infty}_1(\nu_f)$ so that $\|S_f\| \leq 1$ and $\|S_ff\| \geq C^r_f$.
Define \[T:  L^{p,\infty}_r(\mu) \to (\bigoplus_{f\in \Gamma}L^{p,\infty}_1(\nu_f))_\infty \quad\text{by} \quad Tg = (S_fg)_{f\in \Gamma}.\]
Clearly, $T$ is a lattice homomorphism and $\|T\|\leq 1$.
Suppose that $g\in L^{p,\infty}_r(\mu)_+$ with $\|g\|_{L^{p,\infty}_r}=1$.
For any $\ep >0$, there exists a non-negative simple function $h$ such that $g \geq h$ and $\|h\|_{L^{p,\infty}_r} \geq 1-\ep$.
Write $h = \sum^n_{i=1}a_i\chi_{U_i}$.  There exists $I\subseteq \{1,\dots,n\}$ so that
\[\|(\sum_{i\in I}\mu(U_i))^{-\frac{1}{r}+\frac{1}{p}}\sum_{i\in I}a_i\chi_{U_i}\|_r = \|h\|_{L^{p,\infty}_r}  \geq 1-\ep.\]
Then $f : = \sum_{i\in I}a_i\chi_{U_i}\in \Gamma$ with $C_f \geq 1-\ep.$
Thus,
\[ \|Tg\|\geq \|S_fg\| \geq \|S_ff\| \geq C_f^r \geq (1-\ep)^r,
\]
so that $\|Tg\|\geq 1$.  This shows that $T$ is an isometric embedding.
\end{proof}

\section{Subspace problem}\label{Sub problem}
In this section we characterize when an embedding $\iota: F\hookrightarrow E$  induces a lattice embedding $\overline{\iota}: \fbl^{(p,\infty)}[F]\hookrightarrow \fbl^{(p,\infty)}[E]$. The analogous problem for $\fbp$ was solved in \cite[Theorem 3.7]{OTTT} by making use of an extension theorem for regular operators due to Pisier \cite[Theorem 4]{Pisier_REG94}. This scheme of proof, however, does not seem to be  applicable to $\fbl^{(p,\infty)}$. For this reason, we develop in  \Cref{s4.1} an entirely new approach to the subspace problem which is based on push-outs.  Then, in \Cref{s4.2} we prove the injectivity of $\ell^p$ in the class of $p$-convex Banach lattices, which shows that our solution to the subspace problem for $\fbp$ is equivalent to the one in \cite{OTTT}. Finally, in \Cref{s4.3} we prove that $\ell^{p,\infty}$ is \emph{not} injective in the class of Banach lattices with upper $p$-estimates.
\medskip

We begin with some preliminaries. Let $F$ be a closed subspace of a Banach space $E$ and let $\iota:F\to E$ be the inclusion map. Let $q = \iota^*:E^*\to F^*$ and note that $q$ is weak$^*$-to-weak$^*$ continuous with $q(B_{E^*}) = B_{F^*}$. Hence, $f\circ q\in C(B_{E^*})$ if $f\in \FVL[F]$.

\begin{prop}\label{p5.1} 
The map $\overline{\iota}: \FVL[F]\to \FVL[E]$, $\overline{\iota}f = f\circ q$, is a vector lattice isomorphism from $\FVL[F]$ onto the sublattice $\cL$ of $\FVL[E]$ generated by $\{\delta_{\iota y}: y\in F\}$.
\end{prop}

\begin{proof}
Clearly, $\overline{\iota}$ is a lattice homomorphism from $\FVL[F]$ into $C(B_{E^*})$.
Moreover, $(\overline{\iota}\delta_y)(x^*) =(\delta_y\circ q)(x^*) = \delta_{\iota y}(x^*)$ for all $y\in F$ and $x^*\in B_{E^*}$.
Hence, $\overline{\iota}\delta_y = \delta_{\iota y}$.
Thus, $\overline{\iota}\FVL[F]$ is a sublattice of $C(B_{E^*})$ that contains $\delta_{\iota y}$ for all $y\in F$, which implies that $\cL \subseteq \overline{\iota}\FVL[F]$.  On the other hand, $\overline{\iota}^{-1}\cL$ is a sublattice of $\FVL[F]$ containing $\delta_y$ for all $y\in F$.  Hence, $\overline{\iota}^{-1}\cL = \FVL[F]$, which implies that $\overline{\iota}\FVL[F]=\cL$.
\end{proof}


Note that  $\vp_E\iota: F\to \fbl^\cC[E]$ is a contraction (in fact, an isometry) and $\fbl^\cC[E] \in \ol{\cC}$.
By Proposition \ref{p2.1}, $\widehat{\vp_E\iota}: (\FVL[F],\rho_\cC)\to \fbl^\cC[E]$ is a contraction.

\begin{prop}\label{p5.2}
Let $\overline{\iota}$ be the map  from Proposition \ref{p5.1}.
Then $\widehat{\vp_E\iota}f = \overline{\iota}f$ for any $f\in \FVL[F]$.
\end{prop}

\begin{proof}
Let $f\in \FVL[F]$. Suppose that $f = G(\delta_{y_1},\dots, \delta_{y_n})$, where $G$ is  a lattice-linear expression and $y_1,\dots, y_n\in F$. For any $x^*\in B_{E^*}$,
\begin{align*}
(\overline{\iota}f)(x^*) &= f(qx^*) = G(qx^*(y_1),\dots, qx^*(y_n)) = G(x^*(\iota y_1),\dots, x^*(\iota y_n))
\\
& = G(\delta_{ \iota y_1},\dots,\delta_{ \iota y_n})(x^*).
\end{align*}
Thus, $\overline{\iota}f =  G(\delta_{\iota y_1},\dots,\delta_{ \iota y_n})$.
On the other hand,
\[ \widehat{\vp_E\iota}(f) = G(\vp_E\iota y_1,\dots, \vp_E \iota y_n) = G(\delta_{ \iota y_1},\dots, \delta_{\iota y_n}).\]
Therefore, $\overline{\iota}= \widehat{\vp_E\iota}$.
\end{proof}

The operator $\overline{\iota}$ extends to a contraction from  $\fbl^\cC[F]$ into $\fbl^\cC[E]$, which we denote again by $\overline{\iota}$. The objective now is to determine when $\overline{\iota}$ is an embedding. We begin with the following elementary observation.
\begin{prop}\label{p5.3}
Suppose that for any $X\in \cC$, any contraction $T:F\to X$ extends to a contraction $S: E\to X$.
Then $\overline{\iota}$ is a lattice isometric embedding.
\end{prop}

\begin{proof}
In view of the preceding propositions and the fact that $\widehat{\vp_E\iota}$ is a contraction, it suffices to show that $\rho_\cC(f) \leq \rho_\cC(\overline{\iota}f)$ for all $f\in \FVL[F]$.
Let $X\in \cC$ and let $T:F\to X$ be a contraction.
There exists a contraction $S:E\to X$ so that $S \iota = T$.
Suppose that $f = G(\delta_{y_1},\dots, \delta_{y_n})$, where $G$ is  a lattice-linear expression and $y_1,\dots, y_n\in F$. We have
\[ \wh{S}(\overline{\iota}f) = \wh{S}\widehat{\vp_E\iota}(f) = G(S\iota y_1,\dots,S\iota y_n)= G(Ty_1,\dots, Ty_n) = \wh{T}f.
\]
Thus, $\|\wh{T}f\| \leq \rho_\cC(\overline{\iota}f)$.
Taking supremum over all such $T$ shows that $\rho_\cC(f) \leq \rho_\cC(\overline{\iota}f)$.
\end{proof}

Using the above results we may now characterize when $\fbl^\mathcal{C}[F]$ embeds via $\overline{\iota}$ onto a complemented sublattice of $\fbl^\mathcal{C}[E]$. This characterization should be contrasted with \cite[Proposition 3.12]{OTTT} and our solution to the subspace problem in \Cref{SS problem}. Note, in particular,  that if condition (1) in \Cref{t5.4} holds for a class $\mathcal{C}$ then it also holds for any class $\mathcal{C}'$  such that $\mathcal{C}'\subseteq \mathcal{C}$.
\begin{thm}\label{t5.4}
Let $F$ be a closed subspace of a Banach space $E$ and let $\iota : F\to E$ be the inclusion map.
Set $\overline{\iota} = \widehat{\vp_E \iota}$. 
The following are equivalent.
\begin{enumerate}
\item For any $X\in \ol{\cC}$, any bounded linear operator $T:F\to X$ has a bounded linear extension $\ti{T}:E\to X$ with $\|\ti{T}\|=  \|T\|$.
\item $\overline{\iota}$ is a lattice isometric embedding and there is a contractive  lattice homomorphic projection $P$ from $\fbl^\cC[E]$ onto $\overline{\iota}(\fbl^\cC[F])$.
\end{enumerate}
\end{thm}

\begin{proof}
(1)$\implies$(2):  By the above, $\overline{\iota}$ is a contraction. Furthermore, $\overline{\iota} =\widehat{\vp_E\iota}$ is a lattice homomorphism.
Since $\fbl^\cC[F]\in \ol{\cC}$, the isometric embedding $\vp_F: F\to \fbl^\cC[F]$ has a bounded linear extension $\ti{\vp_F}: E\to \fbl^\cC[F]$ such that $\|\ti{\vp_F}\| = 1$.
Therefore,  $\widehat{\ti{\vp_F}}:\fbl^\cC[E] \to \fbl^\cC[F]$ is a lattice homomorphism such that $\|\widehat{\ti{\vp_F}}\| = \|\ti{\vp_F}\|=1$ and  $\widehat{\ti{\vp_F}}\vp_E = \ti{\vp_F}$.
For $y\in F$, it is easy to see that
\[ \widehat{\ti{\vp_F}}\overline{\iota} \delta_y = \widehat{\ti{\vp_F}}\delta_{\iota y} = \ti{\vp_F}(\iota y) = \vp_Fy = \delta_y.\]
Since $\widehat{\ti{\vp_F}}\overline{\iota}$ is a lattice homomorphism, it follows that $\widehat{\ti{\vp_F}}\overline{\iota}$ is the identity map on $\fbl^\cC[F]$.  In particular,
\[\rho_\cC(f) = \rho_\cC(\widehat{\ti{\vp_F}}\overline{\iota}f) \leq  \rho_\cC(\overline{\iota}f) \leq \rho_\cC(f)\text{ for any $f\in \fbl^\cC[F]$.}
\]
Hence, $\overline{\iota}$ is a lattice isometric embedding. Furthermore, $P:= \overline{\iota} \widehat{\ti{\vp_F}}$ is a lattice homomorphic projection on $\fbl^\cC[E]$ such that $P\fbl^\cC[E] = \overline{\iota} \fbl^\cC[F]$ and  $\|P\|=1$.
\medskip

\noindent (2)$\implies$(1):  Let $X\in \ol{\cC}$ and let $T: F\to X$ be a bounded linear operator.
Then $\wh{T}:\fbl^\cC[F] \to X$ is a lattice homomorphism such that $\|\wh{T}\| = \|T\|$ and  $\wh{T}\vp_F = T$.
Let $\ti{T} = \wh{T}\overline{\iota}^{-1}P\vp_E: E\to X$.
For any $y\in F$,
\[ \ti{T}y =  \wh{T}\overline{\iota}^{-1}P\delta_{\iota y} = \wh{T}\overline{\iota }^{-1}\delta_{\iota y} =\wh{T}\delta_y = Ty.
\]
Hence, $\ti{T}$ is a bounded linear extension of $T$.
Clearly, $\|\ti{T}\|\leq  \|\wh{T}\| = \|T\|$.  Hence, $\|\ti{T}\|
= \|T\|$.
\end{proof}
We conclude this preliminary subsection with another simple observation.
\begin{prop}\label{p5.4}
Assume that $\overline{\iota }:\fbl^\cC[F]\to \fbl^\cC[E]$ is a lattice isometric embedding.
Let $X\in \cC$ and let $T:F\to X$ be a bounded linear operator.
The following are equivalent.
\begin{enumerate}
\item There exists an operator $S:E\to X$ such that  $S\iota = T$ and  $\|S\| = \|T\|$.
\item There exists a lattice homomorphism $R: \fbl^\cC[E] \to X$ such that $R\overline{\iota} = \wh{T}$ and  $\|R\| = \|T\|$.
\end{enumerate}
\end{prop}

\begin{proof}
(1)$\implies$(2): Suppose that $S$ is as given.  Then $\wh{S}:\fbl^\cC[E] \to X$ is a lattice homomorphism such that $\|\wh{S}\| = \|S\|= \|T\|$.
If $f= G(\delta_{y_1},\dots, \delta_{y_n})\in \FVL[F]$, then from the proof of Proposition \ref{p5.2}, $\overline{\iota}f = G(\delta_{\iota y_1},\dots, \delta_{\iota y_n})$. Thus,
\[ \wh{S}\overline{\iota}f = \wh{S}G(\delta_{\iota y_1},\dots, \delta_{\iota y_n}) = G(S\iota y_1,\dots, S\iota y_n) = G(Ty_1,\dots, Ty_n)
= \wh{T}f,\]
so (2) holds with $R= \wh{S}$.
\medskip

\noindent(2)$\implies$(1):  Assume that $R$ is as given.
Let $S = R \vp_E: E\to X$. Then $\|S\| = \|R\| = \|T\|$.
Moreover, if $y\in F$ then
\[ S\iota y = R\vp_E\iota y  = R \delta_{\iota y} = R\overline{\iota}\delta_y = \wh{T}\delta_y = Ty.\]
Thus,  (1) holds.
\end{proof}




\subsection{Push-outs in \texorpdfstring{$\mathcal C$}{}}\label{s4.1}

Recall that for objects $A_0,A_1,A_2$ and morphisms $\alpha_i:A_0\rightarrow A_i$, $i=1,2$, a push-out diagram is an object $PO=PO(\alpha_1,\alpha_2)$ together with morphisms $\beta_i:A_i\rightarrow PO$, $i=1,2$, making the following diagram commutative
\begin{center}
\begin{tikzcd}A_1\arrow[r, "\beta_1"]& PO \\A_0\arrow[r, "\alpha_2"]\arrow[u,"\alpha_1"]& A_2\arrow[u, "\beta_2"]\end{tikzcd}
\end{center}
and with the universal property that if $\beta'_i:A_i\rightarrow B$ are such that $\beta'_1\alpha_1=\beta'_2\alpha_2$, then there is a unique $\gamma:PO\rightarrow B$ such that $\gamma\beta_i=\beta'_i$ for $i=1,2,$ as follows:
\begin{center}
\begin{tikzcd}&&B\\ A_1\arrow[r, "\beta_1"]\arrow[bend left, rru, "\beta'_1"]& PO\arrow[ru, "\gamma"] &\\A_0\arrow[r, "\alpha_2"]\arrow[u,"\alpha_1"]& A_2\arrow[u, "\beta_2"]\arrow[bend right, uur, "\beta'_2"]&\end{tikzcd}
\end{center}

In the category of Banach lattices and lattice homomorphisms, isometric push-outs were shown to exist in \cite{Aviles-Tradacete} -- these are push-outs satisfying the extra condition that $\max\{\|\beta_1\|,\|\beta_2\|\}\leq 1$ and which guarantee the inequality $\|\gamma\|\leq \max\{ \|\beta'_1\|,\|\beta'_2\| \}$ in the universal property.
The construction of such push-outs can be adapted to Banach lattices within certain classes $\overline{\mathcal C}$, as follows.

\begin{thm}\label{push-out}
Let $\mathcal C$ be a class of Banach lattices such that $\overline{\mathcal C}$ is closed under lattice quotients. Given Banach lattices $X_0,X_1,X_2$ in $\overline{\mathcal C}$ and lattice homomorphisms $T_i:X_0\rightarrow X_i$ for $i=1,2$, there is a Banach lattice $PO^{\overline{\mathcal C}}$ in  $\overline{\mathcal C}$ and lattice homomorphisms $S_1,S_2$  so that the following is an isometric push-out diagram in $\overline{\mathcal C}$:
\begin{center}
\begin{tikzcd}X_1\arrow[r, "S_1"]& PO^{\overline{\mathcal C}} \\X_0\arrow[r, "T_2"]\arrow[u,"T_1"]& X_2\arrow[u, "S_2"]\end{tikzcd}
\end{center}
\end{thm}

\begin{proof}
Let $X_1\oplus_1 X_2$ denote the direct sum equipped with the norm $\|(x_1,x_2)\|=\|x_1\|+\|x_2\|$ and let $j_i:X_i\rightarrow X_1\oplus_1 X_2$ denote the canonical embedding for $i=1,2$. Let $\phi:X_1\oplus_1 X_2\rightarrow \fbl^\cC[X_1\oplus_1 X_2]$ be the canonical embedding and $Z$ be the (closed) ideal in $\fbl^{\mathcal{C}}[X_1\oplus_1 X_2]$ generated by the families $(\phi(j_1 |x|)-|\phi(j_1 x)|)_{x\in X_1}$, $(\phi(j_2 |y|)-|\phi(j_2 y)|)_{y\in X_2}$ and $(\phi j_1T_1z-\phi j_2T_2z)_{z\in X_0}$. Let
$$
PO^{\overline{\mathcal C}}=\fbl^{\mathcal{C}}[X_1\oplus_1 X_2]/Z,
$$
and let $S_i=q\phi j_i:X_i\rightarrow PO^{\overline{\mathcal C}}$, for $i=1,2$, where $q:\fbl^{\mathcal{C}}[X_1\oplus_1 X_2]\rightarrow PO^{\overline{\mathcal C}}$ denotes the canonical quotient map.
\medskip

Since $\overline{\mathcal C}$ is closed under lattice quotients and $\fbl^{\mathcal{C}}[X_1\oplus_1 X_2]$ is in $\overline{\mathcal C}$, so is $PO^{\overline{\mathcal C}}$. The rest of the proof follows exactly as in \cite[Theorem 4.3]{Aviles-Tradacete}.
\end{proof}

\Cref{push-out} applies, in particular, when $\mathcal{C}$ is the class of $p$-convex Banach lattices or the class of Banach lattices with an upper $p$-estimate (with constant one). Moreover, as in \cite[Theorem 4.4]{Aviles-Tradacete}, for these classes we  have the following.

\begin{thm}\label{parallelisometries}
Let $\mathcal C$ be the class of $p$-convex Banach lattices or Banach lattices with an upper $p$-estimate. Set $K_{\mathcal C}=2^{1-\frac1p}$ if $\cC=\cC_p$ and $K_{\mathcal C}=2^{1-\frac1p}\gamma_p$ if $\cC=\cC_{p,\infty}$, and let
\begin{center}
\begin{tikzcd}X_1\arrow[r, "\widetilde{T}_1"]& PO^{\overline{\mathcal C}} \\X_0\arrow[r, "T_2"]\arrow[u,"T_1"]& X_2\arrow[u, "\widetilde{T}_2"]\end{tikzcd}
\end{center}be an isometric push-out diagram in $\overline{\mathcal C}$. If $\|T_1\|\leq 1$ and the lower arrow $T_2$ is an isometric embedding then the upper arrow $\widetilde{T}_1$ is a $K_{\mathcal C}$-embedding.
\end{thm}

\begin{proof}
Fix $x_1\in X_1$. We aim to  show that $\|x_1\| \leq K_{\mathcal C} \|\widetilde{T}_1(x_1)\|$. The reverse inequality is trivially true since $\|\widetilde{T}_1\|\leq 1$.  Since we are dealing with lattice norms and lattice homomorphisms, we can suppose that $x_1$ is positive. By the push-out universal property, as given in Theorem~\ref{push-out}, it is enough to find a commutative diagram in $\overline{\mathcal C}$
	\begin{center}
		\begin{tikzcd}X_1\arrow[r, "\widehat{T}_1"]& Z \\X_0\arrow[r, "T_2"]\arrow[u,"T_1"]& X_2\arrow[u, "\widehat{T}_2"]\end{tikzcd}
	\end{center}
such that $\|x_1\|\leq \|\widehat{T}_1(x_1)\|$ and $\max\{ \|\widehat{T}_1\|,\|\widehat{T}_2\| \} \leq K_{\mathcal C}$.
\medskip

To this end, by \cite[Theorem 4.4]{Aviles-Tradacete} we have a Banach lattice isometric push-out diagram 
	\begin{center}
		\begin{tikzcd}X_1\arrow[r, "\widetilde{T}_1"]& PO \\X_0\arrow[r, "T_2"]\arrow[u,"T_1"]& X_2\arrow[u, "\widetilde{T}_2"]\end{tikzcd}
	\end{center}
such that $\widetilde T_1$ is an isometric embedding. Hence, by \cite[Lemma 3.3]{Aviles-Tradacete}, we can pick a positive $\sigma$-finite element $x^\ast\in PO^\ast$ of norm one such that $x^\ast(|\widetilde T_1 x_1|) = \|x_1\|$.
\medskip

The functional $x^\ast$ induces a lattice semi-norm on $PO$ given by $\|z\|_{x^\ast} = x^\ast(|z|)$. After making a quotient by the elements of norm 0, this becomes a norm that satisfies $\||x|+|y|\|_{x^\ast} = \|x\|_{x^\ast} +\|y\|_{x^\ast}$ for all $x,y$. This identity extends to the completion of this normed lattice, which, by Kakutani's representation theorem \cite[Theorem 1.b.2]{LT2}, is lattice isometric to a Banach lattice of the form $L^1(\nu_{x^\ast})$. The formal identity induces a homomorphism $\psi_{x^\ast}:PO\rightarrow L^1(\nu_{x^\ast})$ of norm one. 
\medskip

For $i=1,2,$ let $A_i=\psi_{x^\ast}(\widetilde{T}_i(B_{X_i}))$ and set $A=A_1\cup A_2\subseteq L^1(\nu_{x^\ast})$. We will distinguish the following cases:

\begin{enumerate}
    \item Let $\mathcal C=\cC_p$ denote the class of $p$-convex Banach lattices (with constant $1$). Since the $X_i$ are $p$-convex and the $\Psi_{x^*}\widetilde{T}_i$ are contractive lattice homomorphisms we have that for all finitely supported sequences $(\al_i)_{i\in I}$ of real numbers and $(f_i)_{i\in I}\subseteq A$,
\[ \Big\|(\sum |\al_if_i|^p)^{\frac{1}{p}}\Big\|_1 \leq \Big\|(\sum_{f_i\in A_1} |\al_if_i|^p)^{\frac{1}{p}}\Big\|_1+\Big\|(\sum_{f_i\in A_2} |\al_if_i|^p)^{\frac{1}{p}}\Big\|_1\leq 2^{1-\frac{1}{p}}(\sum |\al_i|^p)^{\frac{1}{p}}.\]
    Hence, by \Cref{t3.1}, there is $g\in L^1(\nu_{x^\ast})_+$ with $\|g\|_1\leq1$ and lattice homomorphisms $\widehat{T}_i:X_i\rightarrow L^p(g\nu_{x^\ast})$ that make the following diagram commutative
    \begin{center}
		\begin{tikzcd}X_i\arrow[r, "\widetilde{T}_i"]\arrow[rd, swap, "\widehat{T}_i"]& PO\arrow[r, "\psi_{x^\ast}"] & L^1(\nu_{x^\ast})\\& L^p(g\nu_{x^\ast})\arrow[ru, swap, "M_g"]&\end{tikzcd}
	\end{center}
    with $\|\widehat{T}_i\|\leq 2^{1-\frac1p}=K_{\cC_p}$ and $M_g$ the operator of multiplication by $g$. It follows that
    \begin{center}
		\begin{tikzcd}X_1\arrow[r, "\widehat{T}_1"]& L^p(g\nu_{x^\ast}) \\X_0\arrow[r, "T_2"]\arrow[u,"T_1"]& X_2\arrow[u, "\widehat{T}_2"]\end{tikzcd}
	\end{center}
    is a commutative diagram in $\cC_p$ and
    \[\|x_1\|=x^\ast|\widetilde T_1 x_1|=\|\psi_{x^\ast}\widetilde T_1 x_1\|_{L^1(\nu_{x^\ast})}\leq \|\widehat{T}_1 x_1\|_{L^p(g\nu_{x^\ast})}.\]
    \item Let $\mathcal C=\cC_{p,\infty}$ denote the class of Banach lattices satisfying an upper $p$-estimate (with constant $1$). Since the $X_i$ satisfy an upper $p$-estimate and the $\Psi_{x^*}\widetilde{T}_i$ are contractive lattice homomorphisms we have that for all finitely supported sequences $(\al_i)_{i\in I}$ of real numbers and $(f_i)_{i\in I}\subseteq A$,
\[ \Big\|\bigvee |\al_if_i|\Big\|_1 \leq \Big\|\bigvee_{f_i\in A_1} |\al_if_i|\Big\|_1+\Big\|\bigvee_{f_i\in A_2} |\al_if_i|\Big\|_1\leq 2^{1-\frac{1}{p}}(\sum |\al_i|^p)^{\frac{1}{p}}.\]
    Hence, by  \Cref{t3.2}, there is $g\in L^1(\nu_{x^\ast})_+$ with $\|g\|_1\leq1$ and lattice homomorphisms $\widehat{T}_i:X_i\rightarrow L^{p,\infty}(g\nu_{x^\ast})$ that make the following diagram commutative
    \begin{center}
		\begin{tikzcd}X_i\arrow[r, "\widetilde{T}_i"]\arrow[rd, swap, "\widehat{T}_i"]& PO\arrow[r, "\psi_{x^\ast}"] & L^1(\nu_{x^\ast})\\& L^{p,\infty}(g\nu_{x^\ast})\arrow[ru, swap, "M_g"]&\end{tikzcd}
	\end{center}
    with $\|\widehat{T}_i\|\leq 2^{1-\frac1p}\gamma_p=K_{\cC_{p,\infty}}$ and $M_g$ the operator of multiplication by $g$. It follows that
    \begin{center}
		\begin{tikzcd}X_1\arrow[r, "\widehat{T}_1"]& L^{p,\infty}(g\nu_{x^\ast}) \\X_0\arrow[r, "T_2"]\arrow[u,"T_1"]& X_2\arrow[u, "\widehat{T}_2"]\end{tikzcd}
	\end{center}
    is a commutative diagram in $\cC_{p,\infty}$ and
    \[\|x_1\|=x^\ast|\widetilde T_1 x_1|=\|\psi_{x^\ast}\widetilde T_1 x_1\|_{L^1(\nu_{x^\ast})}\leq \|\widehat{T}_1 x_1\|_{L^{p,\infty}(g\nu_{x^\ast})}.\]
\end{enumerate}


\end{proof}

As a consequence of the above, we get our desired solution to the subspace problem.

\begin{thm}\label{SS problem}
    Suppose that $\iota:F\rightarrow E$ is an isometric embedding and let $\mathcal C$ denote either the class of $p$-convex Banach lattices or the class of Banach lattices with upper $p$-estimates. The following are equivalent.
    \begin{enumerate}
        \item $\overline \iota:\fbl^{\mathcal C}[F]\rightarrow \fbl^{\mathcal C}[E]$ is a $c_1$-lattice embedding.
        \item For every operator $T: F\rightarrow X$ with $X$ in $\mathcal{C}$, there is $Y$ in $\mathcal{C}$, a (norm one) $K_{\mathcal C}$-lattice embedding $j:X\rightarrow Y$ and $S:E\rightarrow Y$ so that $jT=S\iota$ and $\|S\|\leq c_2 \|T\|$.
    \end{enumerate}
    Here, $ c_2\leq  c_1\leq K_{\mathcal C}c_2$.
\end{thm}

\begin{proof}
    (1)$\implies$(2): Let $X\in \mathcal C$ and $T:F\rightarrow X$ an operator. By Theorem \ref{push-out}, let us consider the isometric push-out diagram in $\overline{\mathcal{C}}$ for $X_0=\overline{\iota}(\fbl^{\mathcal{C}}(F))$, $X_1=X$, $X_2=\fbl^{\mathcal{C}}(E)$,  $T_1=\frac{\widehat{T}\overline{\iota}^{-1}}{\|\widehat{T}\overline{\iota}^{-1}\|}:X_0\rightarrow X_1$ and $T_2$ the formal identity from $X_0$ to $X_2$, where $\widehat T:\fbl^{\mathcal C}[F]\rightarrow X$ denotes the lattice homomorphism extending $T$.
    \begin{center}
		\begin{tikzcd}
            & X\arrow[r, "S_1"]& PO^{\overline{\mathcal C}} \\
            \fbl^{\mathcal C}[F]\arrow[r, "\overline \iota"]\arrow[ru,"\frac{\widehat{T}}{\|\widehat{T}\overline{\iota}^{-1}\|}"]& \overline{\iota}(\fbl^{\mathcal{C}}(F)) \arrow[u,"T_1"]\arrow[r,"T_2"] & \fbl^{\mathcal C}[E]\arrow[u, "S_2"]\\
            F \arrow[u,"\phi_F"] \arrow[rr,"\iota"] &  & E \arrow[u,"\phi_E"]
        \end{tikzcd}
	\end{center}
    Let us write $Y=PO^{\overline{\mathcal C}}$ and $j=S_1$. Since $T_2$ is an isometric embedding, we can apply Theorem \ref{parallelisometries} to obtain that $j$ is a norm one $K_{\mathcal C}$-embedding. On the other hand, let $S=\|\widehat{T}\overline{\iota}^{-1}\| S_2\phi_E:E\rightarrow Y$. It follows that
    \[jT= j\wh T \phi_F=\|\widehat{T}\overline{\iota}^{-1}\|j T_1 \ol \iota \phi_F=\|\widehat{T}\overline{\iota}^{-1}\|S_2 T_2 \ol \iota \phi_F=\|\widehat{T}\overline{\iota}^{-1}\|S_2  \phi_E \iota =S\iota,\]
    and $\|S\| \leq \|\widehat{T}\|\|\overline{\iota}^{-1}\|\leq  c_1 \|T\|$, so $c_2\leq c_1 $.\medskip

    \medskip

    \noindent(2)$\implies$(1): First, note that $\|\overline{\iota}\|= 1$. Now, take $f\in \FVL[F]$ and a contraction $T:F\rightarrow X$, $X$ in $\mathcal{C}$. By the hypothesis, there exists $Y$ in $\mathcal{C}$, a norm one $K_{\mathcal C}$-lattice embedding $j:X\rightarrow Y$ and $S:E\rightarrow Y$ so that $jT=S\iota$ and $\|S\|\leq c_2$. It follows that $j\widehat{T}=\widehat{S}\overline{\iota}$, so
    \begin{equation*}
        \|\widehat{T}f\|_X=\|j^{-1}\widehat{S}\overline{\iota}f\|_X\leq \|j^{-1}\|\|S\|\rho_{\mathcal{C}}(\overline{\iota}f)\leq K_{\mathcal C} c_2 \rho_{\mathcal{C}}(\overline{\iota}f)
    \end{equation*}
    and $\overline{\iota}$ is a $K_{\mathcal C} c_2$-embedding.
\end{proof}

\subsection{Recovering the POE-\texorpdfstring{$p$}{} via injectivity}\label{s4.2}
Recall  that a Banach lattice $X$ is \emph{injective} if for every Banach lattice $Z$, every closed linear sublattice $Y$ of $Z$ and every positive linear operator $T:Y\to X$ there is a positive linear extension $\widetilde{T}: Z\to X$ with $\|\widetilde{T}\|=\|T\|.$ Equivalently, $X$ is injective if whenever $X$ lattice isometrically embeds into a Banach lattice $Y$ there is a positive contractive projection from $Y$ onto $X$. Note that this notion of injectivity corresponds to the category of Banach lattices with positive linear maps, which is considerably different from the category of Banach lattices with lattice homomorphisms (where in fact there are no injective objects \cite{Aviles-Tradacete}). The study of injective Banach lattices is classical \cite{MR0383031,MR0473776,MR0383141}; however, it seems that not much is known on injectivity (or projectivity) in subcategories of Banach lattices. Here, we study the injectivity of $\ell^p$ in the class of $p$-convex Banach lattices, with an aim towards showing that \Cref{SS problem} together with \Cref{p5.3} recover \cite[Theorem 3.7]{OTTT}, up to constants. Although the injectivity of $L^p$ spaces in the $p$-convex category is implicit in \cite{Pisier_REG94}, we provide  an independent argument in order to give an alternative proof of \cite[Theorem 3.7]{OTTT}. See also \cite[Corollary 4.13]{2014DH} for a related result.

\begin{thm}\label{t5.9}
Let $X$ be a $p$-convex Banach lattice and suppose that $(x_n)$ is a disjoint positive sequence in $X$ equivalent to the $\ell^p$-basis.  Then $[(x_n)]$ is the image of a positive projection on $X$.
\end{thm}
\begin{proof}
We may assume that $X$ is a closed sublattice of $Z=(\bigoplus_{\mu\in \Gamma}L^p(\mu))_\infty$
by  Proposition \ref{p3.3}.
It suffices to show that $[(x_n)]$ is the range of a positive projection on $Z$.  Assume that $c\|(a_n)\|_p < \|\sum a_nx_n\|\leq \|(a_n)\|_p$ for any non-zero $(a_n)\in \ell^p$.
Let $m\in \N$.  There exists a finite set $\Gamma_m$, $k\in \N$, and a positive contraction $T:(\bigoplus_{\mu\in \Gamma}L^p(\mu))_\infty\to  (\ell^p(k))^{\Gamma_m}_\infty$ so that $(Tx_n)^m_{n=1}$ is a positive disjoint sequence satisfying
$c\|(a_n)^m_{n=1}\|_p < \|\sum^m_{n=1} a_nTx_n\|\leq \|(a_n)^m_{n=1}\|_p$ for any non-zero $(a_n)^m_{n=1}\in \ell^p(m)$. 
Write $Tx_n = (y_n(\gamma))_{\gamma\in \Gamma_m}$ with $y_n(\gamma)\in \ell^p(k)$.
Define $S: \ell^1(m)\to \ell^\infty(\Gamma_m)$ by
\[ S(b_1,\dots, b_m) = (\sum^m_{n=1}b_n\|y_n(\gamma)\|^p)_{\gamma\in \Gamma_m}.
\]
Since $(y_n(\gamma))^m_{n=1}$ is a disjoint sequence for each $\gamma$, if  $b=(b_1,\dots, b_m)\geq 0$, then
\[ \|Sb\| = \max_{\gamma\in \Gamma_m}\|\sum^m_{n=1}b_n^{\frac{1}{p}}y_n(\gamma)\|_p^p = \|\sum^m_{n=1}b_n^{\frac{1}{p}}Tx_n\|^p \geq c^p\|(b_n^{\frac{1}{p}})^m_{n=1}\|_p^p = c^p\|b\|_1.
\]
We claim  that $B^+_{\ell^\infty(m)} \subseteq c^{-p}\co\so\{S^*e^*_\gamma:\gamma\in\Gamma_m\}$. Here, for a subset $A$ of a Banach lattice $X$, $\so A = \bigcup_{a\in A}[-|a|,|a|]$ denotes the solid hull of $A$ and, as usual, $\co(A)$ denotes the convex hull.
To prove the claim, note that if $0\leq y^* \notin c^{-p}\co\so\{S^*e^*_\gamma:\gamma\in\Gamma_m\}$ then by \cite[Proposition 3.1]{OikhbergTursiHulls}
there exists $y\in \ell^1(m)_+$ so that
\[ \|y^*\|\,\|y\| \geq  y^*(y) > c^{-p}\max_{\gamma\in \Gamma_m}(S^*e^*_\gamma)(y) = c^{-p}\|Sy\|\geq \|y\|.
\]
Hence, $\|y^*\| > 1$.  This completes the proof of the claim.
\medskip

By the claim we can find a convex combination $\sum_{\gamma\in \Gamma_m}\al_\gamma S^*e^*_\gamma\geq c^p\overbrace{(1,\dots, 1)}^m$.
Choose $y^*_n(\gamma) \in \ell^{p'}(k)_+$ so that $\supp y^*_n(\gamma) = \supp y_n(\gamma)=: I_n$, $\| y^*_n(\gamma)\|_{p'} = 1$ and
$\la y_n(\gamma), y^*_n(\gamma)\ra = \|y_n(\gamma)\|_p$.
If $x \in (\bigoplus_{\mu\in \Gamma}L^p(\mu))_\infty$ and $Tx = (y(\gamma))_{\gamma\in \Gamma_m}$, let 
\[z^*_n(x) = \sum_{\gamma\in\Gamma_m}\al_\gamma\|y_n(\gamma)\|_p^{p-1}\la y(\gamma), y^*_n(\gamma)\ra.
\]
It is easy to see that $z^*_n$ is a positive linear functional and for any $x\in (\bigoplus_{\mu\in \Gamma}L^p(\mu))_\infty$,
\begin{align*}
\|(z^*_n(x))^m_{n=1}\|_p & \leq \sum_{\gamma\in \Gamma_m}\al_\gamma\bigl\|(\|y_n(\gamma)\|_p^{p-1}\|y(\gamma)\chi_{I_n}\|_p)^m_{n=1}\bigr\|_p\\
&\leq \sup_{\gamma\in\Gamma_m}\sup_n\|y_n(\gamma)\|^{p-1}_p\bigl\|\sum_{n=1}^m y(\gamma)\chi_{I_n}\bigr\|_p \\
&\leq \sup_{\gamma\in\Gamma_m}\sup_n\|y_n(\gamma)\|^{p-1}_p\|y(\gamma)\|_p\\
&\leq \sup_n\|Tx_n\|^{p-1}\cdot \|Tx\|\leq \|Tx\|\leq \|x\|.
\end{align*}
On the other hand, $z^*_l(x_n) = 0$ if $l\neq n$ and
\[z^*_n(x_n) =  \sum_{\gamma\in\Gamma_m}\al_\gamma\|y_n(\gamma)\|_p^{p-1}\la y_n(\gamma), y^*_n(\gamma)\ra = \sum_{\gamma\in\Gamma_m}\al_\gamma\|y_n(\gamma)\|_p^p\geq c^p,
\]
since $\sum_{\gamma\in\Gamma_m}\al_\gamma\|y_n(\gamma)\|_p^p$ is the $n$-th coordinate of the sum $\sum_{\gamma\in \Gamma_m}\al_\gamma S^*e^*_\gamma$.
Therefore, the map $P_m$ on $(\bigoplus_{\mu\in \Gamma}L^p(\mu))_\infty$ given by
\[ P_mx = \sum^m_{n=1}\frac{z_n^*(x)}{z^*_n(x_n)}x_n\]
is a positive projection from  $(\bigoplus_{\mu\in \Gamma}L^p(\mu))_\infty$ onto $[(x_n)^m_{n=1}]$ such that $\|P_m\| \leq c^{-p}$.
Finally, let $\cU$ be a free ultrafilter on $\N$.  Regard $[(x_n)^\infty_{n=1}]$ as a dual Banach lattice  (either the dual of an isomorph of $\ell^{p'}$ if $1< p \leq \infty$ or the dual of an isomorph of $c_0$ if $p=1$).
The operator $P:(\bigoplus_{\mu\in \Gamma}L^p(\mu))_\infty\to [(x_n)]$ given by
$Px = w^*$-$\lim_{m\to \cU}P_mx$ is a positive projection onto $[(x_n)]$.
\end{proof}
As a corollary of Theorems~\ref{SS problem} and \ref{t5.9}, we recover the solution to the subspace problem for $\fbp$ given in \cite[Theorem 3.7]{OTTT}.
\begin{cor}
    Let $\iota:F\hookrightarrow E$ be an isometric embedding and $1\leq p<\infty$. The following are equivalent.
    \begin{enumerate}
        \item $\overline \iota:\fbl^{(p)}[F]\rightarrow \fbl^{(p)}[E]$ is a $c_1$-lattice embedding.
        \item For every operator $T: F\rightarrow \ell^p$ there is $R:E\rightarrow \ell^p$ so that $T=R\iota$ and $\|R\|\leq c_2 \|T\|$.
    \end{enumerate}
    Here, $c_1\leq c_2\leq 2^{p-\frac1p}c_1$.
\end{cor}

\begin{proof}
    We use the identification $\fbl^{(p)}=\fbl^{\mathcal C}$ where $\mathcal C$ denotes the class of $p$-convex Banach lattices and invoke \Cref{SS problem}. The implication $(2)\Rightarrow(1)$ is clear. For the converse, we get that there exist a $p$-convex Banach lattice $Y$, a $2^{1-\frac1p}$-lattice embedding $j:\ell^p\rightarrow Y$ and $S:E\rightarrow Y$ such that $jT=S\iota$ with $\|S\|\leq c_1\|T\|$. By \Cref{t5.9}, there is a (positive) projection $P:Y\rightarrow j(\ell^p)\subseteq Y$ with $\|P\|\leq (2^{1-\frac1p})^p$ (as $j$ is a $2^{1-\frac1p}$-embedding). Thus, we can take $R=j^{-1}PS$.
\end{proof}

\subsection{\texorpdfstring{$\ell^{p,\infty}$}{} is not injective for Banach lattices with upper \texorpdfstring{$p$}{}-estimates}\label{s4.3}
We now  show that for any $1 < p < \infty$, $\ell^{p,\infty}$ is not injective in the class of Banach lattices with upper $p$-estimates.

\begin{thm}\label{not inj}
    For any $1 < p < \infty$  there is a Banach lattice $X$ with an upper $p$-estimate
and an uncomplemented  closed sublattice $Y$ of $X$ that is lattice isomorphic to $\ell^{p,\infty}$.
\end{thm}
\begin{proof}
We begin with some notation. Let $X$ be the Banach lattice $(\bigoplus_\N \ell^{p,\infty})_\infty$.
Clearly, $X$ satisfies an upper $p$-estimate with constant $1$.
We may express any $x\in X$ as $x = (x(j))_{j=1}^\infty = (x(i,j))_{i,j=1}^\infty$, where $(x(j))$ is a bounded sequence in $\ell^{p,\infty}$ and $x(j) = (x(i,j))^\infty_{i=1}$ for each $j$.
Given $m\in \N$, define $P_m$ on $X$ by
\[ (P_mx)(i,j) =\begin{cases} x(m,j)& \text{if $i =m$},\\
0 &\text{if $i\neq m$.}\end{cases}\]
Clearly, $P_m$ is a band projection on $X$ and the projection band $P_mX$ is lattice isometric to $\ell^\infty$ via the map
\[ L_m:\ell^\infty\to P_mX,\ (L_my)(i,j) = \begin{cases}y(j) &\text{if $i = m$,  where $y = (y(j))^\infty_{j=1}$},\\
0&\text{otherwise}.\end{cases}
\]

\medskip

Let $\cS$ be the Schreier family, i.e., all subsets $I$ of $\N$ so that $|I| \leq \min I$.
Since $\cS$ is countable, we can list its elements in a sequence $(S_j)^\infty_{j=1}$.
Given $a= (a_i) \in \ell^{p,\infty}$, let
\[ x_a(i,j) =  \begin{cases}  a_i &\text{if $i\in S_j$,}\\ 0&\text{otherwise.}\end{cases}\]
Denote the sequence of coordinate unit vectors in $\ell^{p,\infty}$ by $(e_i)$. For any $j$,
\[ \|(x_a(i,j))^\infty_{i=1}\|_{p,\infty} = \|\sum_{i\in S_j}a_ie_i\|_{p,\infty} \leq \|(a_i)\|_{p,\infty}.\]
Hence, $T: \ell^{p,\infty}\to X$, $a \mapsto x_a$ is a bounded linear map which is clearly a lattice homomorphism.
On the other hand, if $\|a\|_{p,\infty} >1$ then there exists a non-empty finite set $L$ in $\N$ so that
$|a_i| \geq (p'|L|^{\frac{1}{p}})^{-1}$ for all $i\in L$. 
Let $S_j\in \cS$ be such that $S_j\subseteq L$ and  $|S_j| \geq |L|/2$. Then,
\[ \|x_a\| \geq \|\sum_{i\in S_j}a_ie_i\|_{p,\infty} \geq (p'|L|^{\frac{1}{p}})^{-1}\|\sum_{i\in S_j}e_i\|_{p,\infty} = \frac{1}{p'}\left(\frac{|S_j|}{|L|}\right)^{\frac{1}{p'}} \geq \frac{1}{p'2^{\frac{1}{p'}}}.
\]
This shows that $T$ is a lattice isomorphism from $\ell^{p,\infty}$ into $X$.
\medskip

Let $Y = T\ell^{p,\infty}$.
We will show that there is no bounded projection from $X$ onto $Y$.
Assume on the contrary that there is a projection $P$ from $X$ onto $Y$.
Then there is a sequence $(x^*_i)$ in $X^*$ so that $(x^*_i(x))_i\in \ell^{p,\infty}$ for any $x\in X$ and
$x^*_i(Ta) = a_i$ for any $i$ and $a = (a_k)\in \ell^{p,\infty}$.
It follows readily that $(x^*_i)_i\subseteq X^*$ is equivalent to the unit vector basis of $\ell^{p',1}$.
In particular, it is an unconditional basic sequence in a Banach lattice that is $q$-concave for some $q<\infty$.
Since $(P_i^*)$ is a disjoint sequence of band projections on $X^*$, the sequence $(P^*_ix^*_i)$ is disjoint and hence also unconditional in $X^*$.
By  \cite[Theorem 1.d.6(i)]{LT2}, for any finitely nonzero real sequence $(b_i)$,
\[ \|(b_i)\|_{p',1} \sim \|\sum b_ix^*_i\| \sim \|\sqrt{\sum |b_ix^*_i|^2}\| \geq  \|\sqrt{\sum |b_iP_i^*x^*_i|^2}\| \sim \|\sum b_iP^*_ix^*_i\|.
\]
Thus, the linear map $Q:X\to \ell^{p,\infty}$ defined by $Qx = (P^*_ix^*_i(x))_i$ is bounded. On the other hand, observe that $Te_k \in P_kX$. Therefore,
\begin{equation}\label{01}
    \la Te_k, P_i^*x^*_i\ra = \begin{cases}
\la T e_i, x^*_i\ra = 1 &\text{if $k=i$},\\
0 &\text{otherwise}.\end{cases}
\end{equation}
Hence,  $QTe_i=e_i$ for any $i\in \mathbb{N}$.
\medskip

Regard $L_i$ as a lattice isometric embedding from $\ell^\infty$ into $X$.  Define $y^*_i = L_i^*P_i^*x_i^*$ in $(\ell^\infty)^*$.
Let $I$ be a finite subset of $\N$.  Suppose that $(u_i)_{i\in I}$ is a sequence in $\ell^\infty$ so that $\|(\sum_{i\in I}|u_i|^p)^{\frac{1}{p}}\|_\infty \leq 1$. Set $u = \sum_{i\in I}L_iu_i\in X$.
Then $\|u\| \leq \|(\sum_{i\in I}|u_i|^p)^{\frac{1}{p}}\|_\infty \leq 1$.
Thus,
\begin{align*}\sum_{i\in I}y^*_i(u_i) &= \sum_{i\in I}\la u_i, L^*_iP^*_ix^*_i\ra
= \sum_{i\in I}(P^*_ix^*_i)(u) \\&\leq \|(P^*_ix^*_i(u))\|_{p,\infty}\cdot |I|^{\frac{1}{p'}} = \|Qu\|\cdot |I|^{\frac{1}{p'}}\lesssim |I|^{\frac{1}{p'}}.
\end{align*}
Making use of the duality described on \cite[p.~47]{LT2}, we conclude that there is an absolute constant $C<\infty$ so that
\begin{equation}\label{bound}\|(\sum_{i\in I}|y^*_i|^{p'})^{\frac{1}{p'}}\| \leq C|I|^{\frac{1}{p'}}
\end{equation}
for any finite subset $I$ of $\N$.
Since $(\ell^\infty)^*$ is an $AL$-space, the bounded sequence $(y^*_i)$ has a subsequence $(y^*_{i_k})$ that has a splitting
\[ y^*_{i_k} = u^*_k + v^*_k,\  |u^*_k|\wedge |v^*_k| =0, \text{ $(u^*_k)$ is almost order bounded, $(v^*_k)$ is disjoint}.\]
See, e.g., \cite{MR0937853}.
Since $|v^*_k| \leq |y^*_{i_k}|$ for all $k$, it follows from (\ref{bound}) that $\|v^*_k\| \to 0$.
Thus, the whole sequence $(y^*_{i_k})$ is almost order bounded in $(\ell^\infty)^*$ and hence it is relatively weakly compact.
\medskip

Since $x_i:=Te_i\in P_iX$ and $L_i:\ell^\infty\to P_iX$ is a surjective lattice isometry, we may find a sequence $(u_k)$ in $\ell^\infty$ so that $L_{i_k}u_k = x_{i_k}$ for each $k$.  If we write $u_k = (u_k(j))^\infty_{j=1}$ then
$u_k(j) = 1$ if $k\in S_j$ and $0$ otherwise.
Hence,  for any finitely supported real sequence $(a_k)$,
\[ \|\sum a_ku_k\|_\infty = \sup_j|\sum_{k\in S_j}a_k|.
\]
The expression on the right is the norm of the Schreier space.
The preceding equation says that $(u_k)$ is equivalent to the unit vector basis of the Schreier space.
Thus, $(u_k)$ is a weakly null sequence (in $\ell^\infty$).
Since $\ell^\infty$ has the Dunford-Pettis property, the weakly null sequence $(u_k)$ converges uniformly to $0$ on the relatively weakly compact set $(y^*_{i_k})$.  In particular,
\[ 0 = \lim y^*_{i_k}(u_k) = \lim \la u_k, L^*_{i_k}P^*_{i_k}x^*_{i_k}\ra = \lim \la x_{i_k}, P^*_{i_k}x^*_{i_k}\ra.
\]
However,
\[ e_{i_k} = QTe_{i_k} = Qx_{i_k} = (P^*_jx^*_j(x_{i_k}))_j.\]
Taking the $i_k$-th component yields that $\la x_{i_k}, P^*_{i_k}x^*_{i_k}\ra =1$ for any $k$, which is the desired contradiction.
\end{proof}
A finite dimensional version of \Cref{not inj} can be achieved via a gluing argument. We begin with some preliminaries.
\medskip

By the above, the map $T:\ell^{p,\infty}\to X: = (\oplus_\N \ell^{p,\infty})_\infty$ given by $Ta = (a\chi_{S_j})^\infty_{j=1}$ where $a = (a_i)\in \ell^{p,\infty}$ and $(S_j)$ is the sequence of Schreier sets is a lattice isomorphism so that $T(\ell^{p,\infty})$ is not complemented in $X$.  For each $m\in \N$, regard $\ell^{p,\infty}(m)$ as the subspace $[(e_i)^m_{i=1}]$ in $\ell^{p,\infty}$.
Define $T_m:\ell^{p,\infty}(m)\to X$ by $T_m = T|_{\ell^{p,\infty}(m)}$.
Clearly, there is a constant $C$ so that every $T_m$ is a $C$-lattice embedding.
If $(E_m)$, $(F_m)$ are sequences of Banach spaces and $U_m:E_m\to F_m$ is a bounded linear operator for each $m$ so that $\sup_m\|U_m\| < \infty$, we let $\bigoplus U_m: (\bigoplus E_m)_\infty\to (\bigoplus F_m)_\infty$ be the bounded linear operator defined by $(\bigoplus U_m)((u_m)^\infty_{m=1}) = (U_mu_m)^\infty_{m=1}$.
\begin{cor}
  Let the notation be as above. Suppose that for each $m$, $\pi_m$  is a projection from $X$ onto $Y_m :=T_m(\ell^{p,\infty}(m))$.  Then $\sup_m\|\pi_m\|= \infty$.
\end{cor}
\begin{proof}
Assume, on the contrary, that there is such a sequence $(\pi_m)$ so that $\sup\|\pi_m\| <\infty$. First of all, it is clear that the composition
\[(\bigoplus T_m^{-1})(\bigoplus \pi_m)(\bigoplus T_m) = \bigoplus (T^{-1}_m\pi_m T_m)\]
 is the identity operator on $(\bigoplus \ell^{p,\infty}(m))_\infty$.
Let $\cU$ be a free ultrafilter on $\N$.  If $b = (b_m)\in (\bigoplus \ell^{p,\infty}(m))_\infty$, $b_m = (b_m(j))^m_{j=1}$, let $Qb = (\lim_{m\rightarrow \cU}b_m(j))^\infty_{j=1}$.
It is clear that $Qb\in \ell^{p,\infty}$ and  $Q:  (\bigoplus \ell^{p,\infty}(m))_\infty\to \ell^{p,\infty}$ is a
contraction.
Define $J:\ell^{p,\infty} \to (\bigoplus \ell^{p,\infty}(m))_\infty$ by $Ja = (a\chi_{[1,m]})^\infty_{m=1}$.
Then $J$ is a lattice isometric embedding and  $Q \bigoplus (T^{-1}_m\pi_m T_m)J = QJ$ is the identity map on $\ell^{p,\infty}$.
Finally, define one more map $V: X = (\oplus \ell^{p,\infty})_\infty \to (\oplus X)_\infty$ as follows: for $c = (c_j)^\infty_{j=1} \in  (\oplus \ell^{p,\infty})_\infty$, let
\[ Vc = ((c_j\chi_{[1,m]})^\infty_{j=1})^\infty_{m=1}.
\]
Obviously, $V$ is a contraction.  Moreover, for any $a\in \ell^{p,\infty}$,
\[ VTa = V((a\chi_{S_j})^\infty_{j=1}) = ((a\chi_{S_j}\chi_{[1,m]})^\infty_{j=1})^\infty_{m=1} = (T_m(a\chi_{[1,m]}))^\infty_{m=1} = (\bigoplus T_m)Ja.
\]
In summary, we have the commutative diagram
\begin{center}
		\begin{tikzcd}
		\ell^{p,\infty} \arrow[r, "J"]\arrow[rd, swap, "T"] & (\bigoplus \ell^{p,\infty}(m))_\infty \arrow[r, "\bigoplus T_m"]
		& (\bigoplus X)_\infty \\& X\arrow[ru, swap, "V"] &
		\end{tikzcd}
	\end{center}
and
\[ \ell^{p,\infty}\stackrel{(\bigoplus T_m)J}{\longrightarrow} (\bigoplus X)_\infty \stackrel{Q(\bigoplus T_m^{-1}\pi_m)}{\longrightarrow}\ell^{p,\infty}
\]
is the identity map.
Thus, the map $R:=TQ(\bigoplus T_m^{-1}\pi_m)V:X\to X$ is a projection on $X$ so that $RT = T$.
Hence, $\rng R = T(\ell^{p,\infty})$, contrary to what we previously established for $T$.
\end{proof}

\section*{Acknowledgements}

The research of P.~Tradacete and E.~Garc\'ia-S\'anchez is partially supported by the grants PID2020-116398GB-I00 and CEX2019-000904-S funded by MCIN/AEI/10.13039/501100011033. E.~Garc\'ia-S\'anchez is also partially supported by ``European Union NextGenerationEU/PRT'' and FPI grant CEX2019-000904-S-21-3 funded by MCIN/AEI/10.13039/501100011033. P.~Tradacete is also supported by a 2022 Leonardo Grant for Researchers and Cultural Creators, BBVA Foundation.
\medskip

We thank Timur Oikhberg for interesting discussions related to this work.

\bibliographystyle{plain}
\bibliography{refs.bib}

\end{document}